\documentclass{article} 

\usepackage{enumitem}
\usepackage{nicefrac}
\usepackage[suppress]{color-edits}
\addauthor{zj}{green}
\newcommand\numberthis{\addtocounter{equation}{1}\tag{\theequation}}
\usepackage{prettyref}
\newcommand{\pref}[1]{\prettyref{#1}}

\newcommand{\savehyperref}[2]{\texorpdfstring{\hyperref[#1]{#2}}{#2}}
\newrefformat{eq}{\savehyperref{#1}{Eq.~\textup{(\ref*{#1})}}}
\newrefformat{eqn}{\savehyperref{#1}{Equation~\ref*{#1}}}
\newrefformat{lem}{\savehyperref{#1}{Lemma~\ref*{#1}}}
\newrefformat{lemma}{\savehyperref{#1}{Lemma~\ref*{#1}}}
\newrefformat{def}{\savehyperref{#1}{Definition~\ref*{#1}}}
\newrefformat{line}{\savehyperref{#1}{Line~\ref*{#1}}}
\newrefformat{thm}{\savehyperref{#1}{Theorem~\ref*{#1}}}
\newrefformat{corr}{\savehyperref{#1}{Corollary~\ref*{#1}}}
\newrefformat{cor}{\savehyperref{#1}{Corollary~\ref*{#1}}}
\newrefformat{sec}{\savehyperref{#1}{Section~\ref*{#1}}}
\newrefformat{subsec}{\savehyperref{#1}{Section~\ref*{#1}}}
\newrefformat{app}{\savehyperref{#1}{Appendix~\ref*{#1}}}
\newrefformat{assum}{\savehyperref{#1}{Assumption~\ref*{#1}}}
\newrefformat{ex}{\savehyperref{#1}{Example~\ref*{#1}}}
\newrefformat{fig}{\savehyperref{#1}{Figure~\ref*{#1}}}
\newrefformat{alg}{\savehyperref{#1}{Algorithm~\ref*{#1}}}
\newrefformat{rem}{\savehyperref{#1}{Remark~\ref*{#1}}}
\newrefformat{conj}{\savehyperref{#1}{Conjecture~\ref*{#1}}}
\newrefformat{prop}{\savehyperref{#1}{Proposition~\ref*{#1}}}
\newrefformat{proto}{\savehyperref{#1}{Protocol~\ref*{#1}}}
\newrefformat{prob}{\savehyperref{#1}{Problem~\ref*{#1}}}
\newrefformat{claim}{\savehyperref{#1}{Claim~\ref*{#1}}}
\newrefformat{que}{\savehyperref{#1}{Question~\ref*{#1}}}
\newrefformat{op}{\savehyperref{#1}{Open Problem~\ref*{#1}}}
\newrefformat{fn}{\savehyperref{#1}{Footnote~\ref*{#1}}}
\newrefformat{tab}{\savehyperref{#1}{Table~\ref*{#1}}}
\newrefformat{fig}{\savehyperref{#1}{Figure~\ref*{#1}}}
\newrefformat{proc}{\savehyperref{#1}{Procedure~\ref*{#1}}}
\newrefformat{property}{\savehyperref{#1}{Property~\ref*{#1}}}
\def\nbyp#1{\textcolor{red}{{\bf YP:} #1}}

\newcommand{\gof}{\mathrm{gof}}
\newcommand{\est}{\mathrm{est}}
\newcommand{\stf}{\mathrm{sth}}
\newcommand{\iid}{\mathrm{iid}}
\newcommand{\tv}{\mathrm{TV}}
\newcommand{\iidsim}{\stackrel{\mathrm{i.i.d.}}{\sim}}

\newcommand{\calN}{\mathcal{N}}

\newcommand{\calP}{\mathcal{P}}
\newcommand{\RR}{\mathbb{R}}

\newcommand{\PP}{\mathbb{P}}
\newcommand{\KL}{D_{\mathrm{KL}}}

\newcommand{\EE}{\mathbb{E}}
\newcommand{\lf}{\mathtt{LF}}
\newcommand{\ppx}{\mathbb{P}_{\mathtt{X}}}
\newcommand{\ppy}{\mathbb{P}_{\mathtt{Y}}}
\newcommand{\ppz}{\mathbb{P}_{\mathtt{Z}}}
\newcommand{\px}{p_{\mathtt{X}}}
\newcommand{\py}{p_{\mathtt{Y}}}
\newcommand{\pz}{p_{\mathtt{Z}}}

\newcommand{\bz}{\mathbf{z}}
\newcommand{\bX}{\mathbf{X}}
\newcommand{\bP}{\mathbf{P}}
\newcommand{\bY}{\mathbf{Y}}
\newcommand{\bZ}{\mathbf{Z}}

\newcommand{\card}{\mathrm{card}}
\newcommand{\argmax}{\mathop{\arg\max}}
\renewcommand{\epsilon}{\varepsilon}
\newcommand{\btheta}{{\pmb{\theta}}}
\newcommand{\balpha}{{\pmb{\alpha}}}
\newcommand{\bepsilon}{{\pmb{\epsilon}}}
\newcommand{\hbtheta}{\widehat{\pmb{\theta}}}
\newcommand{\htheta}{\widehat{\theta}}
\newcommand{\zero}{\mathbf{0}}
\newcommand{\simiid}{\stackrel{\mathrm{i.i.d.}}{\sim}}
\newcommand{\var}{\mathrm{Var}}

\newcommand{\ttheta}{\widetilde{\theta}}
\newcommand{\bttheta}{\widetilde{\btheta}}
\newcommand{\tpx}{\widetilde{\px}}
\newcommand{\tpy}{\widetilde{\py}}
\newcommand{\tpz}{\widetilde{\pz}}
\newcommand{\calS}{\mathcal{S}}
\newcommand{\coor}{{\sf{c}}}
\newcommand{\hthetax}{\hat{\theta}_{\mathtt{X}}}
\newcommand{\hthetay}{\hat{\theta}_{\mathtt{Y}}}
\newcommand{\hthetaz}{\hat{\theta}_{\mathtt{Z}}}
\newcommand{\tD}{\tilde{D}}
\newcommand{\be}{\mathbf{e}}
\newcommand{\bv}{\mathbf{v}}
\newcommand{\bu}{\mathbf{u}}
\newcommand{\diam}{\mathrm{diam}}

\def\TV{\mathrm{TV}}
\def\simiid{\stackrel{\text{iid}}{\sim}}

\newcommand\indep{\protect\mathpalette{\protect\independenT}{\perp}}
\def\independenT#1#2{\mathrel{\rlap{$#1#2$}\mkern2mu{#1#2}}}

\usepackage{etoolbox}
\newcommand{\arx}[1]{\iftoggle{aos}{}{#1}}
\newcommand{\aos}[1]{\iftoggle{aos}{#1}{}}
\newtoggle{aos}
\global\togglefalse{aos} 

\arx{
\usepackage{geometry}
\usepackage{amsthm, amsmath, amsfonts, amssymb}
\usepackage[colorlinks,citecolor=blue,urlcolor=blue,linkcolor=blue]{hyperref}
\newtheorem{axiom}{Axiom}
\newtheorem{claim}[axiom]{Claim}
\newtheorem{theorem}{Theorem}[section]
\newtheorem{lemma}[theorem]{Lemma}
\newtheorem{proposition}{Proposition}
\newtheorem{definition}{Definition}
\newtheorem{corollary}{Corollary}
\newtheorem{remark}{Remark}
\geometry{left = 1in, right = 1in, top = 1in, bottom = 1.5in}
}

\aos{
\RequirePackage{amsthm,amsmath,amsfonts,amssymb}
\RequirePackage[numbers,sort&compress]{natbib}
\RequirePackage[colorlinks,citecolor=blue,urlcolor=blue]{hyperref}
\RequirePackage{graphicx}

\startlocaldefs
\theoremstyle{plain}

\newtheorem{theorem}{Theorem}[section]
\newtheorem{lemma}[theorem]{Lemma}
\newtheorem{proposition}{Proposition}
\newtheorem{corollary}{Corollary}
\newtheorem{remark}{Remark}
\theoremstyle{definition}
\newtheorem{definition}[theorem]{Definition}

\endlocaldefs
}

\begin{document}

\aos{
\begin{frontmatter}
}
\title{Testing and estimation in orthosymmetric Gaussian sequence model}

\aos{
\runtitle{Testing and Estimation in Gaussian sequence model}

\begin{aug}
\author[A]{\fnms{Zeyu}~\snm{Jia}\ead[label=e1]{zyjia@mit.edu}}
\and
\author[A]{\fnms{Yury}~\snm{Polyanskiy}\ead[label=e2]{yp@mit.edu}}
\address[A]{Department of EECS, Massachusetts Institute of Technology\printead[presep={ ,\ }]{e1,e2}}

\end{aug}
}

\arx{
    \author{Zeyu Jia
		\\
		\normalsize
        {\texttt{zyjia@mit.edu}}
		\and
		Yury Polyanskiy
		\\
		\normalsize
        {\texttt{yp@mit.edu}}}
    \maketitle
}

\begin{abstract}
We study the Gaussian sequence model, i.e. $X \sim \mathcal{N}(\btheta, I_\infty)$, where $\btheta
\in \Gamma \subset \ell_2$ is assumed to be convex and compact. We show that goodness-of-fit
testing sample complexity is lower bounded by the square-root of the estimation complexity,
whenever $\Gamma$ is orthosymmetric. 
This lower bound is tight when $\Gamma$ is also quadratically
convex (as shown by~\cite{donoho1990minimax,neykov2023signal}). 
We also completely characterize likelihood-free hypothesis testing (LFHT) complexity for
$\ell_p$-bodies, discovering new types of tradeoff between the numbers of simulation and
observation samples, compared to the case of ellipsoids ($p=2$) studied
in~\cite{gerber2024likelihood}.
\end{abstract}

\aos{


\end{frontmatter}
}

\tableofcontents



\section{Introduction}

Two flagship problems in non-parametric statistics are establishing minimax rates (or, equivalently, minimax sample complexities) of density estimation and testing. In the case of density estimation, statistician is given an apriori fixed large class $\Gamma$ of probability distributions  and $n$ iid samples $X_1,\ldots,X_n$ from $P\in \Gamma$ and is tasked with producing an estimate $\hat P$ of the distribution such that $\EE[d(P,\hat P)] \le \epsilon$, where $d$ is a distance and $\epsilon$ is the target accuracy. The minimal number $n$ of samples required for the worst-case choice of $P$ is the sample complexity $n_\est(\Gamma,\epsilon)$. For example, if $\Gamma$ consists of all $\beta$-smooth densities on a compact set in $\mathbb{R}^d$ and if accuracy metric is $d= \|\cdot \|_2$ (the $\ell_2$ distance),
then the classical work of Ibragimov and Khasminskii \cite{ibragimov1983estimation} showed that $n_\est(\Gamma, \epsilon) \asymp \epsilon^{-(2\beta + d) / \beta}$ upto universal constants, cf.~\cite{tsybakov2008introduction}.   

In the case of testing (or goodness-of-fit), statistician is given an apriori fixed large class $\Gamma$ of probability distributions, a member $P_0 \in \Gamma$ and $n$ iid samples $X_1,\ldots,X_n$ from some $P\in \Gamma$, and is tasked with testing hypothesis $P=P_0$ against $d(P,P_0)>\epsilon$ with a fixed small probability of error under either of the alternatives. The minimal number of samples needed for accomplishing this task (worst case over $P$ and $P_0$) is the complexity of testing $n_\gof(\Gamma,\epsilon)$. This way of looking at non-parametric testing, as well as many foundational results, was proposed by Ingster~\cite{ingster1982minimax,ingster1986asymptotic,ingster1987minimax}. In particular, for the Lipschitz class and $\TV$ metric his result shows that $n_\gof(\Gamma,\epsilon) \asymp \epsilon^{-(d/2 + 2)}$. The fact that testing can be done with significantly fewer samples was both surprising and impactful for the development of theoretical statistics in 20th century.

Despite the similarities of two problems, our level of understanding  of $n_\est$ and $n_\gof$ is dramatically different. Specifically, the work of LeCam \cite{le2012asymptotic}, Birge \cite{birge1983approximation, birge1986estimating}, Yatracos \cite{yatracos1985rates}, Yang and Barron \cite{yang1999information} established a direct characterization of $n_\est$ in terms of metric entropy of the class $\Gamma$, thus reducing the problem to that of approximation theory. At the same time, while several powerful tools~\cite{ingster2003nonparametric} were developed for bounding $n_\gof$, its evaluation for each class $\Gamma$ is largely still ad hoc. Thus, our longer term goal is to obtain metric entropy type characterization of $n_\gof$. This work is a step in this direction: by establishing general inequalities relating $n_\gof$ and $n_\est$, implicitly we also obtain entropic bounds on $n_\gof$.

The origin of these kind of relations is the work~\cite{gerber2024likelihood}, which studied yet another statistical problem of \textit{likelihood-free hypothesis testing} (LFHT), see~\eqref{eq: lfht-objective} below, which ``interpolates'' between estimation and testing. By inspecting the minimax region for LFHT at two endpoints~\cite{gerber2024likelihood} observed that for rather different non-parametric types of models (smooth densities, discrete distributions, Gaussian sequence model over ellipsoids) one has a general relationship
\begin{equation}\label{eq:gp-conj}
        \frac{n_\est(\Gamma, \epsilon)}{\epsilon^2} \asymp n_\gof^2(\Gamma, \epsilon)\,.
\end{equation}
This relation can informally be ``derived'' as follows. Up to precision $\epsilon$ the non-parametric model $\Gamma$ can be thought of as a model of finite dimension $D(\Gamma, \epsilon)$, rigorously defined below as a Kolmogorov dimension. Then, classical \textit{parametric} estimation and testing rates allow us to guess
\begin{equation}\label{eq:gof_est_dim}
 n_\est (\Gamma, \epsilon) \asymp \frac{D(\Gamma, \epsilon)}{\epsilon^2}, \qquad 
n_\gof (\Gamma, \epsilon) \asymp \frac{\sqrt{D(\Gamma, \epsilon)}}{\epsilon^2}\,, \end{equation}
which clearly imply~\eqref{eq:gp-conj}. In addition to examples
from~\cite{gerber2024likelihood}, the relationship~\eqref{eq:gp-conj} also holds for a
special class of Gaussian sequence
models with a QCO (compact , convex, quadratically convex and orthosymmetric) constraint set, as
follows from~\cite{donoho1990minimax} and~\cite{neykov2023signal} (see Prop.~\ref{prop:
gof-est-upper} below).

While these considerations support the validity of~\eqref{eq:gp-conj}, unfortunately
neither~\eqref{eq:gof_est_dim} nor~\eqref{eq:gp-conj} hold in general. Indeed, 
a counter-example is implicitly contained in~\cite{baraud2002non}, see~\eqref{eq: def-theta} below.  
Nevertheless, in this work we show that the $\lesssim$ direction of~\eqref{eq:gp-conj} holds under
the mere assumption of orthosymmetry, see~\eqref{eq: gof-est-intro}. 

This paper is written solely in the context of \textit{Gaussian sequence model}, for which $\Gamma$ consists of infinite-dimensional Gaussian densities $\mathcal{N}(\theta, I_\infty)$ with mean $\theta$ constrained to belong to a compact convex set in $\ell_2$. We will abuse notation and simply write $\theta \in \Gamma$ identifying densities with their means. Such models are both simplest to study and are universal: any non-parametric class $\Gamma$ of densities can be shown to be Le Cam equivalent (in the limit of $n\to \infty$) to a certain Gaussian sequence model, cf.~\cite{nussbaum1996asymptotic,brown1996asymptotic}.

\subsection{Our Contribution}

Our contributions are two-fold. First, we derive a collection of results bounding $n_\gof$ in terms of $n_\est$. Second,  we also completely characterize LFHT region of general $\ell_p$-bodies. 

More specifically, for any convex, compact and orthosymmetric set $\Gamma$ (see definitions below), we show in \pref{corr: gof-est} that the minimax sample complexity $n_\gof(\Gamma, \epsilon)$ of goodness-of-fit testing and the minimax sample complexity $n_\est(\Gamma, \epsilon)$ of density estimation satisfy
    \begin{equation}\label{eq: gof-est-intro}
        n_\est(\Gamma, \epsilon) \lesssim \frac{n_\gof^2(\Gamma, \epsilon)}{\epsilon^2} \cdot\mathrm{polylog}(D(\Gamma,\epsilon))\,,
    \end{equation}
    where $D(\Gamma,\epsilon)$ is the Kolmogorov dimension of $\Gamma$ (see also Prop.~\ref{prop: kol-coor-kol}).
    Hence, under rather general conditions half of the  relationship~\eqref{eq:gp-conj} holds up to polylog factors. 
The unusual aspect of our
proof is that a lower bound on testing is shown by extracting a hard to test mixture distribution from the
analysis of a soft-thresholding estimator.
 
   As we mentioned above, if in addition to orthosymmetry one also assumes quadratic-convexity of
   $\Gamma$ (i.e. if $\Gamma$ is QCO), then results of~\cite{donoho1990minimax} on estimation and~\cite{neykov2023signal} on
   testing together imply validity of~\eqref{eq:gp-conj} for such sets (see Prop.~\ref{prop:
   gof-est-upper}), thus showing that our lower bound is generally tight. We recall that
   $\ell_p$-bodies~\cite{baraud2002non,donoho1990minimax} with $p\ge 2$ are QCO sets.

 Are there any counter-examples to~\eqref{eq:gp-conj}? The answer is positive as in fact already implicitly shown in~\cite{baraud2002non}. Specifically, let us define the following 
    \begin{equation}\label{eq: def-theta}
        \Gamma = \left\{\btheta = (\theta_{1:\infty}): \sum_{i\ge 1}i\cdot |\theta_i|\le 1\right\},
    \end{equation}
    which is an example of a more general class of an $\ell_p$-body with $p=1$.
   For this specific class, we find out in \pref{prop: counter-est} and \pref{prop: counter-gof} that
    \begin{equation*} 
    n_\gof(\Gamma, \epsilon) = \tilde{\Theta}\left(\epsilon^{-\frac{12}{5}}\right)\quad \text{and}\quad n_\est(\Gamma, \epsilon) = \tilde{\Theta}\left(\epsilon^{-\frac{8}{3}}\right),
    \end{equation*}
    thus clearly violating~\eqref{eq:gp-conj} and showing that one can indeed have $n_\est \ll n_\gof^2/\epsilon^2$. 


Our second contribution is in establishing minimax rates (regions) for the LFHT problem
defined in~\eqref{eq: lfht-objective}. Recall that in \cite{gerber2024likelihood}, it was shown that the optimal testing region of LFHT for Gaussian sequence model  over an ellipsoid $\Gamma$ satisfies 
    \begin{equation}
    \label{eq:lfht_regular}\left\{m\ge \frac{1}{\epsilon^2}, \quad n\gtrsim n_\gof(\Gamma, \epsilon),\quad \text{and}\quad mn\gtrsim n_\gof^2(\Gamma, \epsilon)\right\}.
    \end{equation}
    It is natural to ask whether this relation is in some sense universal. In this work, we show that for $\Gamma$, which are orthosymmetric, convex and quadratically convex (e.g. $\ell_p$-bodies with $p\ge 2$), the LFHT region remains the same. In particular, this is true for $\ell_p$ bodies with $p\ge 2$. For $p<2$, the general form of the LFHT region is given in terms of an ``effective dimension'' $d(\Gamma, n, \epsilon)$ depending on $n$. Specifically, we show that the testing region is given by 
    $$\left\{m\ge \frac{1}{\epsilon^2}, \quad n\gtrsim \frac{\sqrt{d(\Gamma, n, \epsilon)}}{\epsilon^2},\quad \text{and}\quad mn\gtrsim \frac{d(\Gamma, n, \epsilon)}{\epsilon^4}\right\}.$$
    For example, for the set~\pref{eq: def-theta}, we have $d(\Gamma, n, \epsilon)\asymp \frac{1}{\sqrt{n}\cdot \epsilon^2}$ and the LFHT region is given by 
    $$\left\{m\ge \epsilon^{-2}, n\gtrsim \epsilon^{-\frac{12}{5}},\quad \text{and}\quad m\cdot n^{\frac{3}{2}}\gtrsim \epsilon^{-6}\right\}.$$
    In particular, this shows that the ``regular LFHT'' region, where the boundary is defined in
    terms of the product $mn$ as in \eqref{eq:lfht_regular}, is
    specific to $\ell_p$-bodies with $p\ge 2$, while for $p<2$ the region is rather different.

\subsection{Related Works}
We review some related literatures in this section.

\emph{Non-parametric Density Estimation: }  As we already discussed in the introduction, there has been a long line of work studying the density estimation. For fairly general non-parametric classes and distances between distributions Le Cam \cite{le2012asymptotic} (also in \cite{van2002statistical}) and Birge \cite{birge1983approximation,birge1986estimating} characterized the minimax rate in terms of the (local) Hellinger metric entropy. Other general estimators were proposed by Yang-Barron \cite{yang1999information}, Yatracos \cite{yatracos1985rates} and others. In the context of estimating smooth densities, the study of kernel density estimators is a rich subject \cite[Chapter 1]{tsybakov2008introduction}, as is the method of wavelet-based techniques \cite{donoho1996density}. In the context of Gaussian sequence models, density estimation corresponds to parameter (mean) estimation, with state of the art beautifully summarized in~\cite{johnstone}.

  \emph{Gaussian sequence model: } In the context of Gaussian sequence model density estimation is equivalent to parameter ($\theta$) estimation. This question received significant attention, see~\cite[Chapter 4]{juditsky2020statistical}. We specifically mention pioneering work of Pinsker \cite{pinsker1980optimal}, who demonstrated optimality of linear estimators for the case of certain ellipsoids, and \cite{donoho1990minimax}, who significantly extended this idea by showing optimality (upto a universal factor) of projection estimators for all quadratically convex sets -- a notion which also found application in stochastic optimization \cite{cheng2019geometry}. 

\emph{Goodness-of-Fit Testing: } The sample complexity of goodness-of-fit testing has been pioneered by the already mentioned works of Ingster, whose book~\cite{ingster2012nonparametric} surveys many of the classical developments. Subsequently, \cite{lepski1999minimax} obtained  tight goodness-of-fit testing rate for Besov bodies $B_{s, p, q}(R)$ where $p\in (0, 2)$. 

For Gaussian sequence model, we mention results of~\cite{ermakov1991minimax} and, very relevant for our work, a comprehensive paper of Baraud~\cite{baraud2002non}. Specifically, for $\ell_p$ bodies $\{\theta_{1:D}: \sum_{i=1}^D |\theta_i|^p / a_i^p\le 1\}$ where $a_1\ge a_2\ge \cdots \ge a_D$, they proposed the following dimension
$$\rho(n) = \sup_{d\in [D]} [(\sqrt{d}/n^2\wedge a_d^2\lceil \sqrt{d}\rceil^{1 - 2/p})],$$
and they show that the minimax sample complexity of goodness-of-fit testing is the smallest $n$ such that $\rho(n)\le \epsilon$. More recently, \cite{wei2020local} refined sample complexity to make it depend on a specific (rather than worst-case) choice of the mean in the null-hypothesis. In the special case of 
QCO sets $\Gamma$ Neykov
\cite{neykov2023signal} characterizes (within a constant factor) the goodness-of-fit testing sample complexity in terms ofcritical radius, which in turn is derived from Kolmogorov widths. The case of $\Gamma$ being a $d$-dimensional convex cone was studied in \cite{wei2019geometry}, in particular demonstrating that in the case of ``ice-cream cones'' one can get $n_\gof \asymp {1\over \epsilon^2}$ but $n_\est \asymp {d\over \epsilon^2}$, due to null-case being at the apex of the cone.

A large amount of work has also been done on the topic in computer science literature under different names of \emph{identity testing} or \emph{uniformity testing}, see \cite{goldreich1998property, batu2000testing, paninski2008coincidence, valiant2017automatic, valiant2020instance, canonne2021identity, canonne2023full}. A recent lower bound for robust testing was proposed in \cite{diakonikolas2017statistical}. \cite{canonne2021random, diakonikolas2023gaussian} proposed minimax optimal testing scheme for cases with unknown variances. An excellent monograph \cite{canonne2022topics} surveys this line of work. 


\emph{Likelihood Free Hypothesis Testing: } The form of likelihood free hypothesis testing was firstly introduced in \cite{gutman1989asymptotically, ziv1988classification}. They studied the problem in fixed finite alphabet. \cite{zhou2020second} showed that the testing scheme introduced in \cite{gutman1989asymptotically} is second-order optimal. This problem is extended into sequential and distributional setting in \cite{hsu2020binary, he2020distributed,haghifam2021sequential, boroumand2022universal}. In this work, we will focus on the setting introduced in \cite{gerber2024likelihood}.






\subsection{Organization}
In \pref{sec: preliminaries} we review the basic concepts of goodness-of-fit testing, density estimation and likelihood-free hypothesis testing. In \pref{sec: gof-est}, we
study the relationship between goodness-of-fit testing and density estimation. In \pref{sec: lfht} we study the feasible region of likelihood-free hypothesis testing. Specifically, in \pref{sec: lfht-density} we build up relationship between the LFHT feasible region and density estimation, and in \pref{sec: lfht-ell-p} we calculate the LFHT feasible region for $\ell_p$ bodies with $p\le 2$.

\subsection{Notations} 
For $\btheta\in \RR^d$, we use $\calN(\btheta, I_d)$ to denote the multivariate Gaussian distribution with mean $\btheta$ and variance matrix to be the identity matrix. We use $\zero$ to denote the all-zero vector. We use $a_n = \mathcal{O}(b_n)$ or $a_n\lesssim b_n$ ($a_n = \Omega(b_n)$ or $a_n\gtrsim b_n$) to denote the inequality $a_n\le c\cdot b_n$ ($a_n\ge c\cdot b_n$) for all $n$ for some fixed positive constant $c$.

\section{Preliminaries and notation}\label{sec: preliminaries}


We review some basic concepts of the Gaussian sequence model \cite{ingster2012nonparametric}, goodness-of-fit testing \cite{ingster2012nonparametric}, density estimation \cite{johnstone} and likelihood-free hypothesis testing \cite{gerber2024likelihood} in this section.

\vspace{0.2cm}
\textbf{Gaussian Sequence Model. } We focus on unit-variance multivariate Gaussian location model, which is specified by an integer $D \in [1,\infty]$ and a subset $\Gamma \subseteq \mathbb{R}^D$, so that the model consists of all distributions $\calP(\Gamma) \triangleq \{\calN(\btheta, I_D): \btheta\in \Gamma\}$. When $D=\infty$ we also refer to this model as Gaussian sequence model. In the following, we often use subset $\Gamma$ to denote the model $\calP(\Gamma)$ itself.


\vspace{0.2cm}
\textbf{Goodness of Fit Testing. } Given an integer $D\in [1, \infty]$, we conduct the goodness of fit testing in $\RR^D$. In this task, the statistician is given an class of parameters $\Gamma\subseteq \RR^D$, and has the knowledge that the true parameter of the model $\btheta\in \Gamma$. However, $\btheta$ is unknown and the statistician can only obtain information of $\btheta$ through $n$ samples $\bX = (\bX_{1:n})\sim \ppx^{\otimes n}$ where $\ppx = \calN(\btheta, I_D)$. The statistician conducts the following hypothesis testing problem
\begin{equation}\label{eq: goodness-of-fit}
    H_0: \btheta = \zero\quad \text{versus}\quad H_1: \|\btheta\|_2\ge \epsilon,
\end{equation}
i.e., based on the samples in $\bX$, the statistician makes choice $\psi(\bX)\in \{0, 1\}$, where $\psi(\bX) = 0$ denotes the statistician accepts hypothesis $H_0$ and $\psi(\bX) = 1$ denotes the statistician rejects $H_0$. We focus on the smallest number of $n$ such that 
\begin{equation}\label{eq: condition-gof}
    \max_{i\in \{0, 1\}} \sup_{P\in H_i} \bP(\psi(\bX) \neq i)\le \frac{1}{4},
\end{equation}
where $\sup_{P\in H_0}$ denotes the case $\btheta = \zero$, and $\sup_{P\in H_1}$ denotes the supreme over $\btheta$ with $\|\btheta\|_2\ge \epsilon$. We let $n_\gof(\Gamma, \epsilon)$ to denote the smallest number of samples such that \pref{eq: condition-gof} holds.

\vspace{0.2cm}
\textbf{Density Estimation. } Given an integer $D\in [1, \infty]$, we conduct the density estimation task in $\RR^D$. In this task, the statistician is given an class of parameters $\Gamma\subseteq \RR^D$, and has the knowledge that the parameter of the groundtruth distribution $\btheta\in \Gamma$. However, $\btheta$ is unknown and the statistician can only obtain information of $\btheta$ through $n$ samples $\bX = (\bX_{1:n})\sim \ppx^{\otimes n}$ where $\ppx = \calN(\btheta, I_D)$. Based on the samples in $\bX$, the statistician proposes an estimator $\hbtheta(\bX)\in \RR^D$. We focus on the smallest number of samples $n$ such that there exists an estimator which achieves expected $\ell_2$ estimation error no more than $\epsilon^2$ for any distribution in the distribution class, i.e.
\begin{equation}\label{eq: density-estimation}
    \inf_{\hbtheta}\sup_{\btheta\in \Gamma} \EE\left[\left\|\hbtheta(\bX) - \btheta\right\|_2^2\right]\le \epsilon^2.
\end{equation}
We use $n_\est(\Gamma, \epsilon)$ to denote the smallest number of samples such that \pref{eq: density-estimation} holds. 

\vspace{0.2cm}
\textbf{Likelihood Free Hypothesis Testing. } The task of likelihood free hypothesis testing was first introduced in \cite{gerber2024likelihood}. The statistician conducts the testing in $\RR^D$, and the statistician is given the class $\Gamma$ of the Gaussian sequence model. Suppose there are two artificial densities $\ppx = \calN(\btheta^1, I_D)$, $\ppy = (\btheta^2, I_D)$ and a true density $\ppz = \calN(\btheta, I_D)$ where $\btheta^1, \btheta^2, \btheta\in \Gamma$ are unknown to the statistician. After collecting $n$ samples from $\ppx$ and $\ppy$ each to form artificial datasets $\bX\sim \ppx^{\otimes n}$ and $\bY\sim \ppy^{\otimes n}$, and collecting $m$ samples from $\ppz$ to form the ground truth dataset $\bZ\sim \ppz^{\otimes m}$, the statistician conducts the following hypothesis testing problem:
\begin{equation}\label{eq: lfht}
    H_0: \btheta = \btheta^1\quad\text{versus}\quad H_1: \btheta = \btheta^2.
\end{equation}
To characterize the minimax sample complexity of the above hypothesis testing problem, we suppose $\btheta^1$ and $\btheta^2$ satisfies $\|\btheta^1 - \btheta^2\|_2\ge \epsilon$ for some $\epsilon > 0$, and this piece of information is given to the statistician. After receiving data $(\bX, \bY, \bZ)\in (\RR^D)^n\times (\RR^D)^n\times (\RR^D)^m$, the statistician makes choice $\psi(\bX, \bY, \bZ)\in \{0, 1\}$, where $\psi(\bX, \bY, \bZ) = 0$ denotes the statistician accepts hypothesis $H_0$ and $\psi(\bX, \bY, \bZ) = 1$ denotes the statistician rejects $H_0$. We focus on the conditions of $(m, n)$ such that 
\begin{equation}\label{eq: lfht-objective}
    \max_{i\in \{0, 1\}} \sup_{P\in H_i} \PP(\psi(\bX, \bY, \bZ) \neq i)\le \frac{1}{4},
\end{equation}
where $\sup_{P\in H_0}$ denotes the supreme over $\btheta^1$ and $\btheta^2$ with $\|\btheta^1 - \btheta^2\|_2\ge \epsilon$ and $\btheta = \btheta^1$, and $\sup_{P\in H_1}$ denotes the supreme over $\btheta^1$ and $\btheta^2$ with $\|\btheta^1 - \btheta^2\|_2\ge \epsilon$ and $\btheta = \btheta^2$. 

For any pair $(m, n)$, we say that $(m, n)$ lies in the feasible region of the likelihood-free hypothesis test if and only if there exists some test scheme $\psi$ such that \pref{eq: lfht-objective} holds.

\vspace{0.2cm}
\textbf{$\ell_p$ Bodies. } Given an integer $D\in [1, \infty]$, the $\ell_p$ bodies in $\RR^D$ is characterized by a non-increasing nonnegative sequence $a_1\ge a_2\ge \cdots \ge a_D\ge 0$.  The $\ell_p$ body (also appears in \cite{baraud2002non}) characterized by $a_{1:D}$ is given by
\begin{equation}\label{eq: ell-p-body}
    \Gamma = \left\{\theta = (\theta_1, \theta_2, \cdots, \theta_D)\ \Big|\ \sum_{t=1}^D \frac{|\theta_t|^p}{a_t^p}\le 1\right\}\subseteq \RR^D.
\end{equation}
When $D = \infty$, in order to guarantee the compactness of this infinite-dimensional set, we only consider the $\ell_p$ bodies characterized by sequence $a_{t}$ goes to zero, i.e. $\lim_{t\to \infty} a_t = 0$.

\vspace{0.2cm}
\textbf{Orthosymmetric Sets. } We recall the definition of orthosymmetric sets, first introduced in \cite{donoho1990minimax}: Given an integer $D\in [1, \infty]$, we say $\Gamma\subseteq \RR^D$ is an orthosymmetric set if for any element $\btheta = (\theta_1, \theta_2, \cdots, \theta_D)\in \Gamma$ and $\bepsilon = (\epsilon_1, \epsilon_2, \cdots, \epsilon_D)\in \{-1, 1\}^D$, we have 
$$\btheta_{\bepsilon} = (\epsilon_1\theta_1, \epsilon_2\theta_2, \cdots, \epsilon_D\theta_D)\in \Gamma.$$
We notice that many sets satisfy the orthosymmetric property. It is easy to verify that the $\ell_p$ bodies introduced in \pref{eq: ell-p-body} are orthosymmetric. Additionally, if there exists a function $f: (\RR_+\cup\{0\})^D\to \RR$ such that $\Gamma$ can be represented as 
    \begin{equation}\label{eq: symmetric-set}
        \Gamma = \left\{\btheta = (\theta_1, \theta_2, \cdots, \theta_D): f(|\theta_1|, |\theta_2|, \cdots, |\theta_D|)\ge 0\right\},
    \end{equation}
    then $\Gamma$ is an orthosymmetric set. It is easy to see that all $\ell_p$ bodies can be written in the above form for some function $f$.

\vspace{0.2cm}
\textbf{Kolmogorov Dimension and Coordinate-wise Kolmogorov Dimension. } Given an integer $D\in [1, \infty]$, the Kolmogorov dimension of a set $\Gamma\subseteq \RR^D$ is a notion to measure the minimal dimension of an affine space so that the distance between any point in $\Gamma$ and the affine space is bounded by some tolerance. This notion is given formally as in the following definition. 
    \begin{definition}[Kolmogorov Dimension]\label{def: kolmogorov}
        Suppose $\Gamma\subseteq \RR^D$. For any $\epsilon > 0$, we define the Kolmogorov dimension $D(\Gamma, \epsilon)$ of $\Gamma$ at scale $\epsilon$ to be the largest integer $d\le D$ such that 
        $$\inf_{\Pi_d} \sup_{\btheta\in \Gamma} \|\btheta - \Pi_d\btheta\|_2\le \epsilon,$$
        where $\Pi_d$ denotes a $d$-dimensional linear projection, and the infimum is over all possible $d$-dimensional linear projections.
    \end{definition}
    Notice that in the above definition, the projection operator can be chosen to be any linear projections. However, in some cases, we have to restrict ourselves to project only along the coordinates. In this way, we define the following coordinate-wise Kolmogorov dimension.
    \begin{definition}[Coordinate-wise Kolmogorov Dimension] \label{def: kolmogorov-coor} 
        Suppose $\Gamma\subseteq \RR^D$. For any $\epsilon > 0$, we define the coordinate-wise Kolmogorov dimension $D_\coor(\Gamma, \epsilon)$ of $\Gamma$ at scale $\epsilon$ to be the largest integer $d\le D$ such that 
        $$\inf_{A:|A|=d} \sup_{\btheta\in \Gamma} \sum_{i\in \mathbb{Z}_+\backslash A}(\theta_i)^2\le \epsilon^2,\qquad \text{where } \btheta = (\theta_1, \theta_2, \cdots, \theta_D),$$
        where the infimum is over all size-$d$ subsets $A$ of $\{1, 2, \cdots, D\}$.
    \end{definition}

    It is clear that for any set $\Gamma$ and $\epsilon > 0$, $D_\coor(\Gamma, \epsilon)\ge D(\Gamma, \epsilon)$. However, it is possible to have $D_\coor \gg D$. Fortunately, our bounds depend on $D_\coor$ only logarithmically, and hence the following pair of results can be used to roughly relate $D_\coor$ to $D$ and $\epsilon$.
    \begin{proposition}
    For any orthosymmetric, convex, compact set $\Gamma$, we have 
    $$D\left(\Gamma, \frac{\epsilon}{D_\coor(\Gamma, \epsilon)}\right)\ge \frac{D_\coor(\Gamma, \epsilon)}{2} - 2.$$
    \end{proposition}
    \begin{corollary}
    Suppose orthosymmetric, convex, compact set $\Gamma$ satisfies $D(\Gamma, \epsilon)\lesssim \epsilon^{-p}$ for some $0 < p < 1$. Then the coordinate-wise Kolmogorov dimension satisfies 
    $$D_\coor(\Gamma, \epsilon)\lesssim \epsilon^{-p/(1-p)}.$$
    \end{corollary}


\section{Goodness-of-Fit and Density Estimation}\label{sec: gof-est}

We return to the testing-estimation relation~\eqref{eq:gp-conj} that was shown in~\cite{gerber2024likelihood} to hold for a variety of models, including Gaussian sequence models with $\Gamma$ being an ellipsoid. Our first goal (\pref{sec: gof-est-upper}) is to exhibit a natural Gaussian sequence model with convex $\Gamma$, for which the sample complexity of goodness-of-fit testing is \(\tilde{\Theta}(\varepsilon^{-12/5})\), while the sample complexity of density estimation is \(\tilde{\Theta}(\varepsilon^{-8/3})\), thus showing that~\eqref{eq:gp-conj} fails. Next, in \pref{sec: gof-est-lower} we show that nevertheless, for orthosymmetric convex $\Gamma$ a one-sided comparison always holds: $\epsilon^2 n_{\gof}^2 \gtrsim n_{\est}$. Finally, in \pref{sec: app-gof-est-quadratic} we show that two-sided relationship~\eqref{eq:gp-conj} in fact holds for all orthosymmetric, compact, convex and quadratically-convex $\Gamma$.


\subsection{A Counter Example}\label{sec: gof-est-upper}
We consider the set defined in \pref{eq: def-theta}: for $D = \infty$, and 
$$\Gamma = \left\{\btheta = (\theta_{1:\infty}):\quad \sum_{i\ge 1}i\cdot |\theta_i|\le 1\right\}\subseteq \RR^D.$$
Then we have the following characterization in the goodness-of-fit testing sample complexity and density estimation sample complexity.
\begin{proposition}\label{prop: counter-est}
    For set $\Gamma$ defined in \pref{eq: def-theta}, the minimax density estimation sample complexity $n_\est(\Gamma, \epsilon)$ satisfies 
    $$n_\est(\Gamma, \epsilon) = \tilde{\Theta}\left(\epsilon^{-8/3}\right).$$
\end{proposition}
\begin{proposition}\label{prop: counter-gof}
    For set $\Gamma$ defined in \pref{eq: def-theta}, the minimax goodness-of-fit testing sample complexity $n_\gof(\Gamma, \epsilon)$ satisfies 
    $$n_\gof(\Gamma, \epsilon) = \tilde{\Theta}\left(\epsilon^{-12/5}\right).$$
\end{proposition}
The proof of \pref{prop: counter-est} and \pref{prop: counter-gof} are deferred to \pref{sec: app-gof-est-upper}. From \pref{prop: counter-est} and \pref{prop: counter-gof}, we learn that for this specific set $\Gamma$, \pref{eq:gp-conj} cannot hold when $\epsilon$ goes to zero, even up to log factors this conjecture still fails.

\subsection{One-side inequality for Orthosymmetric Convex Sets}\label{sec: gof-est-lower}
Even though there exists a set $\Gamma$ such that \pref{eq:gp-conj} fails, we show that a one-side inequality between the minimax sample complexity of goodness-of-fit testing and the minimax sample complexity of density estimation holds up to log factors, if we assume orthosymmetric property, compactness and convexity of the set $\Gamma$. This result is summarized in the following theorem.
\begin{theorem}\label{thm: gof-est}
    For positive integer $D$, suppose $\Gamma\in \RR^D$ is an orthosymmetric, compact and convex set (orthosymmetric sets defined in \pref{sec: preliminaries}). Then we have 
    $$n_\gof\left(\Gamma, \frac{\epsilon}{\sqrt{2}}\right)^2\ge \frac{1}{64\log\left(2D\right)}\cdot \frac{n_\est(\Gamma, \epsilon)}{\epsilon^2}.$$
\end{theorem}
The proof of \pref{thm: gof-est} is deferred to \pref{sec: app-gof-est-lower}. \pref{thm: gof-est} has the following implication on the infinite-dimensional orthosymmetric convex sets. 
Then we have the following corollary, which generalize \pref{thm: gof-est} into infinite dimensional sets. The following corollary lower bounds $n_\gof$ in terms of the coordinate-wise Kolmogorov dimension of the set $\Gamma$ (defined in \pref{def: kolmogorov-coor}).
\begin{corollary}\label{corr: gof-est}
    Suppose the coordinate-wise Kolmogorov dimension of $\Gamma$ at scale $\epsilon$ to be $D_\coor(\Gamma, \epsilon)$. Then if $\Gamma$ is compact, convex and orthosymmetric, we have for any $\epsilon > 0$,
    $$n_\gof\left(\Gamma, \frac{\epsilon}{\sqrt{2}}\right)^2\ge \frac{1}{64\log(2D_\coor(\Gamma, \epsilon))}\cdot \frac{n_\est(\Gamma, \sqrt{3}\epsilon)}{\epsilon^2}.$$
\end{corollary}
The proof of \pref{corr: gof-est} is deferred to \pref{sec: app-gof-est-lower}.

\begin{remark}
    It is easy to see that all $\ell_p$ bodies defined in \pref{sec: preliminaries} are orthosymmetric, compact and convex sets. Hence \pref{thm: gof-est} (for finite dimensional sets), and \pref{corr: gof-est} (for infinite dimensional sets) holds for all $\ell_p$ bodies.
\end{remark}

\begin{remark}
    The notion of coordinate-wise Kolmogorov dimension $D_\coor(\Gamma, \epsilon)$ (defined in \pref{def: kolmogorov-coor}) differs from the traditional Kolmogorov dimension $D(\Gamma, \epsilon)$ (defined in \pref{def: kolmogorov}), as the projection spaces in $D_\coor$ are restricted to be parallel or perpendicular to the coordinate axes. It is clear that $D_\coor(\Gamma, \epsilon) \ge D(\Gamma, \epsilon)$. However, in some cases, $D_\coor(\Gamma, \epsilon)$ can also be upper-bounded in terms of $D(\Gamma, \epsilon)$. For further discussion, see \pref{sec: coor-kolmogorov}.
\end{remark}

    \begin{remark}
        We notice that without assuming the orthosymmetric property, \pref{thm: gof-est} can fail.
	For example, ~\cite{wei2019geometry} shows that for the $d$-dimensional circular cone with
	null-hypothesis at the apex of the cone the goodness-of-fit testing complexity is
	$n_\gof\asymp 1/\epsilon^2$ (independent to the dimension $d$), while the estimation
	complexity is the usual $n_\est \asymp d / \epsilon^2$. Based on this idea, a
	non-orthosymmetric counter-example to the lower bound in the Theorem can be constructed by taking 
    $$\Gamma = \left\{\btheta = (\theta_1, \theta_2, \cdots): \sum_{i=2}^\infty\theta_i^2\le \theta_1^2, \theta_1\ge 0\right\} \cap \left\{\btheta = (\theta_1, \theta_2, \cdots): \sum_{i=1}^\infty i^2\cdot \theta_i^2\le 1\right\}\,.$$
    Indeed, it can be shown similarly to~\cite{wei2019geometry} that $n_\gof\asymp \epsilon^{-2}$, while $n_\est\gtrsim \epsilon^{-3}$. 
    
    \end{remark}

\subsection{Testing-Estimation Equivalence for Quadratically Convex Set}\label{sec: gof-est-quadratic}
In the above, we already verify that the lower bound side of \pref{eq:gp-conj} holds, if we assume orthosymmetric property, compactness and also convexity of the set $\Gamma$. We wonder in what circumstances can the opposite side of inequality \pref{eq:gp-conj} hold. In the following, we show that if we additionally assume that quadratically convexity of set $\Gamma$, then the opposite side of \pref{eq:gp-conj} also holds. This result, together with \pref{thm: gof-est} (or \pref{corr: gof-est}), shows that the equivalence relation \pref{eq:gp-conj} holds up to log factors, if assuming that the set is orthosymmetric, compact, convex and quadratically convex. Before presenting this result, we first recap the definition of quadratically convex sets (first introduced in \cite{donoho1990minimax}).
\begin{definition}[Quadratically Convex Set]\label{def: quadratically}
    We say a set $\Gamma\subseteq \RR^D$ is quadratically convex, if the following set is convex:
    $$\left\{\btheta^2: \btheta\in \Gamma\right\},$$
    where $\btheta^2 = (\theta_1^2, \cdots, \theta_D^2)$ for $\btheta = (\theta_1, \cdots, \theta_D)\in \Gamma$. 
\end{definition}
\begin{remark}
    The main result in \cite{donoho1990minimax} was the proof of optimality (upto a universal constant factor $1.25$) of linear estimators (also known as projection estimators) for the minimax estimation rate in the Gaussian location model over an orthosymmetric, compact, convex and quadratically convex class. 
\end{remark}

We have the following theorem for quadratically convex sets.
\begin{proposition}[\cite{neykov2023signal,donoho1990minimax}]\label{prop: gof-est-upper}
    There exists universal constants $c, C > 0$ such that for any QCO (compact, convex,
    quadratically convex and orthosymmetric) set $\Gamma$,
    $$n_\gof(\Gamma, c\epsilon)^2\lesssim \frac{n_\est(\Gamma, \epsilon)}{\epsilon^2}\lesssim n_\gof(\Gamma, C\epsilon)^2,\qquad \forall \epsilon > 0.$$
\end{proposition}
\begin{proof} Seminal work \cite{donoho1990minimax} shows that for QCO sets, $D(\Gamma, c\epsilon) /
\epsilon^2\lesssim n_\est(\Gamma, \epsilon)\lesssim D(\Gamma, C\epsilon) / \epsilon^2$. 
In \cite{neykov2023signal} it is shown that 
 $n_\gof(\Gamma, c\epsilon)\lesssim \sqrt{D(\Gamma, \epsilon)} / \epsilon^2\lesssim
n_\gof(\Gamma, C\epsilon)$ for QCO sets. These results together imply the proposition.
\end{proof}

In \pref{sec: app-gof-est-quadratic} we also include a slightly more general result showing that
the second inequality in the Proposition holds under the weaker assumption on $\Gamma$ (namely,
only requiring that projection estimators be order-optimal).

\begin{remark}
    As shown in \cite{donoho1990minimax}, $\ell_p$-bodies with $p\ge 2$ are quadratically convex sets.
\end{remark}

\section{Likelihood Free Hypothesis Testing}\label{sec: lfht}
In this section, we study the feasible region of likelihood free hypothesis testing (LFHT) problems, which is introduced in \cite{gerber2024likelihood}. We already reviewed the basic of LFHT in \pref{sec: preliminaries}.

Earlier in \pref{sec: gof-est-lower} we established a lower bound for the goodness-of-fit testing sample complexity in terms of the density estimation sample complexity (for compact, convex and orthosymmetric sets). Under additional assumption of quadratic convexity (\pref{def: quadratically}), we have shown (\pref{sec: gof-est-quadratic}) an upper bound matching the lower bound up to log factors.

\subsection{LFHT and Density Estimation}\label{sec: lfht-density}
In this section we generalize both of these bounds to the setting of LFHT. \pref{thm: lfht-orthosymmetric}  establishes the lower bound (for compact, convex and orthosymmetric sets), while \pref{thm: quad-lfht} shows a matching upper bound (under additional assumption of quadratic convexity). As a corollary, this resolves (up to log factors) characterization of the LFHT region of $\ell_p$ bodies with $p\ge 2$. 

\begin{theorem}\label{thm: lfht-orthosymmetric}
    Suppose $\Gamma\subseteq \RR^D$ is an orthosymmetric, compact and convex set. Then if a pair of integers $(m, n)$ lies in the LFHT region, i.e. there exists testing scheme $\psi: (\RR^D)^n\times (\RR^D)^n\times (\RR^D)^m$ such that \pref{eq: lfht-objective}, then $(m, n)$ satisfies 
    \begin{align*} 
      & m\ge \frac{1}{\epsilon^2},\qquad n\ge \frac{\sqrt{n_\est(\Gamma, \sqrt{2}\epsilon)}}{8\epsilon\cdot \sqrt{\log(4D)}}\\
      &\qquad \text{and}\quad mn\ge \frac{n_\est(\Gamma, \sqrt{2}\epsilon) - 1}{3072\epsilon^2\cdot \log(4D)}.
    \end{align*}
\end{theorem}
The proof of \pref{thm: lfht-orthosymmetric} can be found in \pref{sec: lfht-orthosymmetric-app-l}. The above theorem provide a lower bound to the feasible region of LFHT. Next, we provide an upper bound to the feasible region of LFHT, in terms of the Kolmogorov dimension. We introduce the following testing scheme: For $\epsilon > 0$, let $\Pi_d$ to be the $d$-dimensional linear projection satisfying
\begin{equation}\label{eq: quadratically-convex-kolmogorov}
    \sup_{\btheta\in \Gamma}\|\btheta - \Pi_d\btheta\|_2\le \frac{\epsilon}{3},
\end{equation}
where $d = D(\Gamma, \epsilon/3)$ denotes the Kolmogorov dimension of set $\Gamma$ at scale $\epsilon/3$ (\pref{def: kolmogorov}). Consider samples $\bX = (\bX_1, \cdots, \bX_n)\simiid \ppx$, $\bY = (\bY_1, \cdots, \bY_n)\simiid \ppy$ and $\bZ = (\bZ_1, \cdots, \bZ_n) \simiid \ppz$. Letting
$$\hthetax = \frac{1}{n}\sum_{i=1}^n \bX_i, \quad \hthetay = \frac{1}{n}\sum_{t=1}^n \bY_i\quad \text{and}\quad \hthetaz = \frac{1}{n}\sum_{t=1}^n \bZ_i,$$
we define the testing function
\begin{equation}\label{eq: quad-lhft}
    T_\lf(\bX, \bY, \bZ) = \left\|\Pi_d[\hthetax - \hthetaz]\right\|^2 - \left\|\Pi_d[\hthetay - \hthetaz]\right\|^2,
\end{equation}
and consider the testing scheme: $\psi(\bX, \bY, \bZ) = \mathbb{I}[T_\lf\ge 0]$. The following theorem indicates that that this testing scheme satisfies \pref{eq: lfht-objective}, as long as $m, n$ lower bounded by some functions of the Kolmogorov dimension $D(\Gamma, \epsilon/3)$.
\begin{theorem}\label{thm: quad-lfht} 
    If the pair of integers $(m, n)$ satisfies
    \begin{equation}\label{eq: quad-lfht-region}
        m\ge \frac{96}{\epsilon^2},\quad n\ge \frac{96\sqrt{D(\Gamma, \epsilon/3)}}{\epsilon^2}\quad \text{and}\quad mn\ge \frac{768D(\Gamma, \epsilon/3)}{\epsilon^4},
    \end{equation}
    the testing scheme $\psi(\bX, \bY, \bZ) = \mathbb{I}\{T_{\lf}\ge 0\}$ with $T_{\lf}$ defined in \pref{eq: quad-lhft} satisfies \pref{eq: lfht-objective}.
\end{theorem}

The proof of \pref{thm: quad-lfht} is deferred to \pref{sec: quad-uncondintional-app}. The above upper bound characterization is with respect to the Kolmogorov dimension. We wonder the relationship between the feasible region of the above testing scheme and the minimax sample complexity of goodness-of-fit testing or density estimation with $\Gamma$. With the help of the relations in \cite{donoho1990minimax} between the minimax density estimation sample complexity, and the Kolmogorov dimension, we have the following direct corollary.
\begin{corollary}\label{corr: quad-lfht}
    Suppose $\Gamma$ is an uncondintional, compact, convex and quadratically convex set, then there exists a universal positive constant $c_0$ such that if the pair of integers $(m, n)$ satisfies 
    $$m\ge \frac{c_0}{\epsilon^2},\quad n\ge \frac{c_0\sqrt{n_\est(\Gamma, \epsilon/9)}}{\epsilon}\quad \text{and}\quad mn\ge c_0^2\cdot n_\est(\Gamma, \epsilon/9)^2,$$
    the testing scheme $\psi(\bX, \bY, \bZ) = \mathbb{I}\{T_{\lf}\ge 0\}$ with $T_{\lf}$ defined in \pref{eq: quad-lhft} satisfies \pref{eq: lfht-objective}.
\end{corollary}
This corollary directly follows from \pref{thm: quad-lfht} and the lower bound  of the minimax density estimation sample complexity from the Kolmogorov dimension for sets which are orthosymmetric, compact, convex and quadratically convex. We next apply \pref{thm: lfht-orthosymmetric} and \pref{corr: quad-lfht} to $\ell_p$ bodies with $p\ge 2$ (recall \pref{eq: ell-p-body}), which are orthosymmetric, compact, convex and quadratically convex from \cite{donoho1990minimax}. Hence we have the following characterization.
\begin{corollary}[LFHT for $\ell_p$-body with $p\ge 2$]
    For $\ell_p$ bodies given in \pref{eq: ell-p-body}, we define
    $$\tD(\epsilon) = \min\left\{n\in \mathbb{Z}_+: \forall\ \btheta\in \Gamma,\ \sum_{j\in [D]}(\theta_j)^2\wedge \frac{1}{n}\le \epsilon^2\right\}.$$
    Then if there exists a testing scheme $\psi$ such that \pref{eq: lfht-objective} holds, then $(m, n)$ satisfies 
    $$m\gtrsim \frac{1}{\epsilon^2},\quad n\gtrsim \sqrt{\frac{\tD(\sqrt{2}\epsilon)}{\epsilon^2\log(4D)}}\quad\text{and}\quad mn\gtrsim \frac{\tD(\sqrt{2}\epsilon)}{\epsilon^2\log(4D)},$$
    and if $(m, n)$ satisfies 
    $$m\gtrsim\frac{1}{\epsilon^2},\quad n\gtrsim \sqrt{\frac{\tD(\epsilon/3)}{\epsilon}}\quad\text{and}\quad mn\gtrsim \frac{\tD(\epsilon/3)}{\epsilon^2},$$
    then there exists a testing scheme $\psi$ such that \pref{eq: lfht-objective} holds.
\end{corollary}
The proof amounts to applying \pref{thm: lfht-orthosymmetric} and \pref{corr: quad-lfht} after noticing that $n_\est(\epsilon) \asymp \tD(\epsilon)$ for all all $\ell_p$-bodies, cf. \cite[Chapter 4]{johnstone}. 

\subsection{Tight characterization of LFHT for $\ell_p$ Bodies with $p\le 2$}\label{sec: lfht-ell-p}
From the last section, we have both sufficient and necessary conditions of the feasible region of LFHT, in terms of the density estimation rate. These characterization are tight for $\ell_p$ bodies where $p\ge 2$. However, when $p\le 2$, the set $\Gamma$ is no longer quadratically convex, hence \pref{corr: quad-lfht} fails. In this section, we provide tight characterization for $\ell_p$ bodies where $p\le 2$.

Firstly, we recall the form of the testing region in \cite{gerber2024likelihood}:
$$\left\{(m, n):\qquad m\gtrsim\frac{1}{\epsilon^2},\quad n\gtrsim n_\gof(\epsilon)\quad \text{and}\quad mn\gtrsim n_\gof(\epsilon)^2\right\},$$
and this work proposed a conjecture that for any set the testing region is always in the above form. While the results in the previous section does not violates this conjecture, we wonder whether this conjecture holds for $\ell_p$ bodies where $p\le 2$. Surprisingly, it turns out the testing region for $\ell_p$ bodies is not in the above form. This can be seen from the following theorem.
\begin{theorem}[Informal]    For given positive integer $n$ and set of parameters $\Gamma$ given in \pref{eq: ell-p-body}, we define function $d(\Gamma, n, \epsilon)$ as
    $$d(\Gamma, n, \epsilon) = \max \left\{d: (a_d)^pn^{\frac{p-2}{2}}\gtrsim \epsilon^2\right\}.$$
    Then the LFHT region for $\Gamma$ at scale $\epsilon$ is given by 
    \begin{equation}\label{eq: lfht-region-lp}
        \left\{(m, n): \qquad m\gtrsim \frac{1}{\epsilon^2},\quad n\gtrsim \frac{\sqrt{d(\Gamma, n, \epsilon)}}{\epsilon^2}\quad\text{and}\quad mn\gtrsim \frac{d(\Gamma, n, \epsilon)}{\epsilon^4}\right\}.
    \end{equation}
\end{theorem}
In the following, we will first present the results for finite dimensions (in \pref{sec: lfht-u} and \pref{sec: lfht-l}), and later we will generalize the results to infinite dimensional sets in \pref{sec: lfht-infinite}. At the end of this section, as an example, we revisit the set $\Gamma$ defined in \pref{eq: def-theta}, and we present the closed form of the feasible testing region of this set.

\subsubsection{Testing Scheme and Upper Bounds}\label{sec: lfht-u}
We first let $\delta = 1/32$ and define 
\begin{equation}\label{eq: def-d-u}
    d_u(\Gamma, n, \epsilon) = \max \left\{d: (a_d)^pn^{\frac{p-2}{2}}\gtrsim \frac{\epsilon^2}{576\log(4D/\delta)}\right\}.
\end{equation}
Suppose $(m, n)$ satisfies 
\begin{equation}\label{eq: lfht-region-ub}
    \left\{(m, n): \qquad m\ge \frac{32}{\epsilon^2},\quad n\ge \frac{32\sqrt{d_u(\Gamma, n, \epsilon)}}{\epsilon^2}\quad\text{and}\quad mn\ge \frac{512d_u(\Gamma, n, \epsilon)}{\epsilon^4}\right\}.
\end{equation}
In this section we will come up with a testing scheme $\Psi$ such that \pref{eq: lfht-objective} is achieved. First of all, for cases where $m > n$, if we let $m_0 = n$, then $(m_0, n)$ is also in the above region, and also $m_0\le m$. Hence without loss of generality, we only need to construct a testing scheme for $m\le n$.

Next, when $mn\ge 1024d_u(\Gamma, n, \epsilon)/\epsilon^2$, if $m\ge 32\sqrt{d_u(\Gamma, n, \epsilon)}/\epsilon^2$, we let $m_0 = n_0 = 32\sqrt{d_u(\Gamma, n, \epsilon)}/{\epsilon^2}$, then $m_0\le m$ and $n_0\le n$, and $(m_0, n_0)$ is in the region defined in \pref{eq: lfht-region-ub}, hence we only need to construct testing for $(m_0, n_0)$. If $m\le 32\sqrt{d_u(\Gamma, n, \epsilon)}/\epsilon^2$, by letting $n_0 = \lceil 512d_u(\Gamma, n, \epsilon)/(\epsilon^2 m)\rceil$ we will have $m\le n_0\le n$, and $(m, n_0)$ satisfies $mn_0\le 1024d_u(\Gamma, n, \epsilon)/(\epsilon^2 m)$. We only need to construct testing for $(m, n_0)$. Therefore, without loss of generality, we assume
\begin{equation}\label{eq: condition-mn}
    mn \le \frac{1024 d_u(\Gamma, n, \epsilon)}{\epsilon^4}.
\end{equation}

We introduce a testing scheme for $\ell_p$ bodies in this section. Inspired by the soft-thresholding algorithms for density estimation, which are known to be minimax optimal for $\ell_p$ balls where $p\in [1, 2]$ \cite{johnstone} and also the goodness of fit testing algorithms for $\ell_p$ bodies \cite{baraud2002non}, we design an algorithm for the setting of likelihood-free hypothesis testing. To describe the testing regime, we use $\bX = (\bX_1, \cdots, \bX_n)$ to denote the $n$ i.i.d. samples collected from $\ppx$, $\bY = (\bY_1, \cdots, \bY_n)$ to denote the $n$ i.i.d. samples collected from $\ppy$ and $\bZ = (\bX_1, \cdots, \bZ_m)$ to denote the $m$ i.i.d. samples collected from $\ppz$. 

To begin with, we divide those $n$ samples $(\bX_1, \cdots, \bX_n)$ from $\ppx$ and samples $\bY_1, \cdots, \bY_n$ from $\ppy$ each into two parts: let $n_0 = \lfloor n/2\rfloor$, and
\begin{align*}
    \bX^1 = (\bX_1, \cdots, \bX_{n_0}),\quad & \bX^2 = (\bX_{n_0+1}, \cdots, \bX_{n}),\\
    \text{and}\quad \bY^1 = (\bY_1, \cdots, \bY_{n_0}),\quad & \bY^2 = (\bY_{n_0+1}, \cdots, \bY_{n}).
\end{align*}
We use $\hthetax^1, \hthetax^2, \hthetay^1, \hthetay^2$ and $\hthetaz$ to denote the average of samples within $\bX^1, \bX^2, \bY^1, \bY^2$ and $\bZ$, i.e. 
\begin{align*} 
    \hthetax^1 & = \frac{1}{n_0}\sum_{i=1}^{n_0} \bX_i,\quad \hthetax^2 = \frac{1}{n_0}\sum_{i=n_0+1}^n \bX_i,\quad \hthetay^1 = \frac{1}{n_0}\sum_{i=1}^{n_0} \bY_i,\\
    \hthetay^2 & = \frac{1}{n_0}\sum_{i=n_0+1}^n \bY_i,\quad \text{and}\quad \hthetaz = \frac{1}{m}\sum_{j=1}^m \bZ_j.
\end{align*}
The testing involves the following two steps: 
\begin{enumerate}[label=(\roman*)]
    \item Calculating a subspace of small dimension such that $\ppx$ and $\ppy$ are different within the subspace;
    \item Adopting the likelihood free hypothesis testing scheme in \cite[(2.2)]{gerber2024likelihood} within the subspace.
\end{enumerate}
We will illustrate these two steps in detail:
\begin{enumerate}[label=(\alph*)]
    \item \label{item: coordinates}\textbf{Calculating the Subspace: } Our goal in this step is to use samples from $\bX^1$ and $\bY^1$ to calculate a subset $T\subseteq [D]$ such that 
    \begin{enumerate}[label=(\roman*)]
        \item Limit size: $|T|\lesssim d_u(\Gamma, n, \epsilon)$;
        \item Enough Separation: $\sum_{t\in T}\left((\px)_t - (\py)_t\right)^2\gtrsim \epsilon^2$.
    \end{enumerate}
    To achive this, we construct the following set $T_1, T_2\subseteq [D]$ which satisfies $T_1\cap T_2 = \emptyset$ and $T_1\cup T_2 = [D]$:
    \begin{equation}\label{eq: def-T-1-T-2}
        T_1 = \{1, 2, \ldots, d_u(\Gamma, n, \epsilon)\}\quad \text{and}\quad T_2 = \{d_u(\Gamma, n, \epsilon)+1, d_u(\Gamma, n, \epsilon)+2, \ldots, D\},
    \end{equation}
    where $d_u(\Gamma, n, \epsilon)$ is defined in \pref{eq: def-d-u}. For $t\in [D]$, we suppose the $t$-th coordinate of $\hthetax^1$ and $\hthetay^1$ to be $(\hthetax^1)_t$ and $(\hthetay^1)_t$. We further define 
    \begin{equation}\label{eq: def-T-3}
        T_3 = \left\{t:\quad t\in T_2\ \ \text{and} \ \ |(\hthetax^1)_t - (\hthetay^1)_t|\ge 4\sqrt{\frac{2\log(2D/\delta)}{n}}\right\}.
    \end{equation}
    for some parameter $\delta > 0$ to be defined later. And we construct $T = T_1\cup T_3$.
    \item \textbf{Testing in the Subspace: } After obtaining the set $T$ of coordinates the last step, we use the samples within $\bX^2, \bY^2$ and $\bZ$ to do likelihood free hypothesis testing within the subspace formed through coordinates in $T$. The testing scheme is similar to \cite[(2.2)]{gerber2024likelihood}, i.e. if we use $(\hthetax^2)_t$, $(\hthetay^2)_t$ and $(\hthetaz)_t$ to denote $t$-th coordinates of $\hthetax^2$, $\hthetay^2$ and $\hthetaz$, then we define 
    \begin{equation}\label{eq: stf-test}
        T_{\lf} = \sum_{t\in T}\left[\left((\hthetax^2)_t - (\hthetaz)_t\right)^2 - \left((\hthetay^2)_t - (\hthetaz)_t\right)^2\right].
    \end{equation} 
    And we adopt the testing scheme $\psi(\bX, \bY, \bZ) = \mathbb{I}\{T_{\lf}\ge 0\}$.
\end{enumerate}

We argue that the testing scheme obtained in the above steps is minimax optimal up to constants and log factors. This result is summarized in the following theorem.
\begin{theorem}\label{thm: lp-lfht-ub}
    For $1\le p\le 2$, if $(m, n)$ lies in the region \pref{eq: lfht-region-ub}, the test $\psi(\bX, \bY, \bZ) = \mathbb{I}\{T_{\lf}\ge 0\}$ with $T_{\lf}$ defined in \pref{eq: stf-test} satisfies \pref{eq: lfht-objective}.
\end{theorem}
The proof of \pref{thm: lp-lfht-ub} is deferred to \pref{sec: app-lfht-u}. 

\subsubsection{Lower Bounds}\label{sec: lfht-l}
In this section, we show that the region in \pref{eq: lfht-region-lp} is also necessary for the testing. We first define 
\begin{equation}\label{eq: def-d-l}
    d_l(\Gamma, n, \epsilon) = \max\left\{d: (a_d)^p n^{\frac{p-2}{2}}\ge 192\epsilon^2\right\}.
\end{equation}

This is summarized in the following theorem.
\begin{theorem}\label{thm: lower-bound-lfht}
    If there exists a testing scheme $\psi$, which takes $n$ samples from $\px$ and $\py$ and $m$ samples from $\pz$, and $\psi$ satisfies \pref{eq: lfht-objective}, then $(m, n)$ satisfies 
    \begin{equation}\label{eq: lb-lfht}
        m\ge \frac{1}{\epsilon^2}, \quad n\ge \frac{\sqrt{d_l(\Gamma, n, \epsilon)}}{2\epsilon^2},\quad \text{and}\quad mn\ge \frac{d_l(\Gamma, n, \epsilon)}{96\epsilon^4}.
    \end{equation}
\end{theorem}
The proof of \pref{thm: lower-bound-lfht} is deferred to \pref{sec: app-lfht-l}.

\subsubsection{Comparison with results of Baraud}
In \cite{baraud2002non}, the minimax sample complexity of goodness-of-fit testing for $\ell_p$ bodies is given. According to \cite[Section 4.1]{baraud2002non} and \cite[Section 4.2]{baraud2002non}, for fixed positive $n$ and $\ell_p$ bodies $\Gamma$ defined in \pref{eq: def-theta}, if we let 
$$d = \argmax_{d} \left\{\frac{\sqrt{d}}{n}\wedge a_d^2\cdot d^{1 - 2/p}\right\},$$
then the goodness-of-fit testing can be done if and only if $\epsilon^2\gtrsim \frac{\sqrt{d}}{n}$. Hence the goodness-of-fit testing region at scale $\epsilon$ satisfies 
$$n\gtrsim \frac{\sqrt{d(\Gamma, n, \epsilon)}}{\epsilon^2},$$
where $d(\Gamma, n, \epsilon)$ is given by
$$d(\Gamma, n, \epsilon) = \max\left\{d: (a_d)^p n^{\frac{p-2}{2}}\ge \epsilon^2\right\}.$$
We notice that this form of $d(\Gamma, n, \epsilon)$ is similar to the form $d_u(\Gamma, n, \epsilon)$ and $d_l(\Gamma, n, \epsilon)$ defined in \pref{eq: def-d-u} and \pref{eq: def-d-l}.

\subsubsection{Generalization to Infinite Dimensional Sets}\label{sec: lfht-infinite}
In this section, we generalize \pref{thm: lp-lfht-ub} and \pref{thm: lower-bound-lfht} to infinite dimensional $\ell_p$ bodies. For a given positive non-increasing sequence $a_1\ge a_2\ge \cdots\ge a_n$, we consider set 
\begin{equation}\label{eq: theta-l-p-bodies}
    \Gamma = \left\{\btheta = (\theta_{1:\infty}): \sum_{t=1}^\infty \frac{|\theta_t|^p}{a_t^p}\le 1\right\}.
\end{equation}
To guarantee the compactness of the set $\Gamma$, we assume that $\lim_{t\to \infty} a_t = 0$. To begin with, we recall the definition of the coordinate-wise Kolmogorov dimension in \pref{def: kolmogorov-coor}. For this set, we can easily check that 
\begin{equation}\label{eq: coor-ell-p-bodies}
    D_\coor(\Gamma, \epsilon) = \min \left\{D\ge 1: a_D\le \epsilon\right\}.
\end{equation}
Then we can characterize the feasible region of likelihood-free hypothesis testing in the following theorems: 
\begin{theorem}[Upper Bounds for Infinite Dimensional Sets]\label{thm: lfht-infinite-u}
    For set $\Gamma$ in the form \pref{eq: theta-l-p-bodies}, we let the Kolmogorov dimension of $\Gamma$ at scale $\epsilon$ to be $D(\Gamma, \epsilon)$. We further define 
    $$d_u(\Gamma, n, \epsilon) = \max\left\{d: (a_d)^p n^{\frac{p-2}{2}}\ge \frac{\epsilon^2 / 9}{576 \log(4D_\coor(\Gamma, \epsilon/3)/\delta)}\right\}.$$
    Then if $(m, n)$ satisfies
    $$\left\{(m, n): \qquad m\ge \frac{32}{\epsilon^2},\quad n\ge \frac{32\sqrt{d_u(\Gamma, n, \epsilon)}}{\epsilon^2}\quad\text{and}\quad mn\ge \frac{512d_u(\Gamma, n, \epsilon)}{\epsilon^4}\right\},$$
    then if we adopt the testing scheme $\psi$ defined in \pref{sec: lfht-u} with $D = D_\coor(\Gamma, \epsilon/3)$, then $\psi$ satisfies \pref{eq: lfht-objective}.
\end{theorem}

\begin{theorem}[Lower Bounds for Infinite Dimensional Sets]\label{thm: lfht-infinite-l}
    For given $\Gamma$, we define 
    $$d_l(\Gamma, n, \epsilon) = \max\left\{d: (a_d)^p n^{\frac{p-2}{2}}\ge 192\epsilon^2\right\}.$$
    Then any $(m, n)$ such that there exists a test $\psi$ which satisfies \pref{eq: lfht-objective} must satisfy
    $$m\ge \frac{1}{\epsilon^2}, \quad n\ge \frac{\sqrt{d_l(\Gamma, n, \epsilon)}}{2\epsilon^2},\quad \text{and}\quad mn\ge \frac{d_l(\Gamma, n, \epsilon)}{96\epsilon^4}.$$
\end{theorem}
The proof of \pref{thm: lfht-infinite-u} and \pref{thm: lfht-infinite-l} are deferred to \pref{sec: app-lfht-infinite}.

\subsubsection{Examples}\label{sec: lfht-theta}
We revisit the set defined in \pref{eq: def-theta}. In this section, we will calculate the feasible region of likelihood-free hypothesis testing for set $\Gamma$, which is summarized in the following proposition.
\begin{proposition}\label{prop: lfht-theta}
    The feasible region of $(m, n)$ of likelihood-free hypothesis testing for set $\Gamma$ defined in \pref{eq: def-theta} contains the following set:
    $$\left\{(m, n):\quad m\ge \epsilon^{-2},\ n\gtrsim \epsilon^{-\frac{12}{5}}\log^{2/5}(1/\epsilon),\ m\cdot n^{\frac{3}{2}}\gtrsim \epsilon^{-6}\log(1/\epsilon)\right\},$$
    and is contained by the following set:
    $$\left\{(m, n):\quad m\ge \epsilon^{-2},\ n\gtrsim \epsilon^{-\frac{12}{5}},\ m\cdot n^{\frac{3}{2}}\gtrsim \epsilon^{-6}\right\},$$
    where we use $\gtrsim$ to hide universal constant factors.
\end{proposition}
The proof of \pref{prop: lfht-theta} directly follows from \pref{thm: lfht-infinite-u} and \pref{thm: lfht-infinite-l} after noticing that $D(\Gamma, \epsilon) = 1/\epsilon$ and
$$d_u(\Gamma, n, \epsilon)\lesssim \frac{\log(1/\epsilon)}{\sqrt{n}\cdot \epsilon^2},\quad \text{and}\quad d_l(\Gamma, n, \epsilon)\gtrsim \frac{1}{\sqrt{n}\cdot \epsilon^2}.$$
Note that this proposition also provide a hard case where the conditions of likelihood-free hypothesis testing region do not hold in \cite[Open Problem 4]{gerber2024likelihood}.

\section*{Acknowledgement}

Authors would like to thank S. Balakrishnan, M. Neykov and Y. Wei for pointing out several
important results in the literature.
We would also like to thank Oleg Lepski and Martin Wainwright for many discussions and
comments.



\bibliographystyle{alpha}
\bibliography{ref_aos}

\newpage
\appendix
\section{Analysis between Density Estimation and Goodness-of-fit Testing}

\subsection{Proofs of \pref{prop: counter-est} and \pref{prop: counter-gof}}\label{sec: app-gof-est-upper}

\begin{proof}[Proof of \pref{prop: counter-est}]
    Our proof is divided into two parts: the lower bound to the density estimation sample complexity and the upper bound to the density estimation sample complexity.
    
    \vspace{0.2cm}
    \textbf{Lower Bound to Density Estimation: } 
    Without loss of generality we assume $(2\epsilon)^{-2/3}$ is an integer (otherwise we replace $\epsilon$ by $\lfloor (2\epsilon)^{-2/3}\rfloor^{-3/2}/2$ and the argument follows only up to constant). We first construct a infinite-dimensional rectangle which is a subset of $\Gamma$: 
    $$M = \left\{\btheta = (\theta_1, \theta_2, \cdots): |\theta_i|\le (2\epsilon)^{4/3},\ \forall 1\le i\le (2\epsilon)^{-2/3}, \quad \text{and}\quad \theta_i = 0,\ \forall i\ge d\right\}.$$
    We can verify that for any $\btheta = (\theta_1, \theta_2, \cdots)\in M$, 
    $$\sum_{i=1}^\infty i\cdot |\theta_i|\le \sum_{i=1}^{(2\epsilon)^{-2/3}} i\cdot (2\epsilon)^{4/3}\le 1.$$
    Hence $\btheta\in \Gamma$. Therefore, $M\subseteq \Gamma$. We next lower bound the density estimation error of estimation using $n$ samples: according to \cite[Proposition 4.16]{johnstone}, we have 
    $$\inf_{\hbtheta_n}\sup_{\btheta\in M}\EE\left[\|\hbtheta_n - \btheta\|^2\right] = (2\epsilon)^{-2/3}\cdot \inf_{\htheta_n}\sup_{|\theta|\le (2\epsilon)^{4/3}} \EE\left[(\htheta_n - \theta)^2\right],$$
    where $\hbtheta_n$ denotes an estimator with $n$ i.i.d. samples coming from $\calN(\btheta, I)$, and $\htheta_n$ denotes an estimator with $n$ i.i.d. samples coming from $\calN(\theta, 1)$. Next, according to Van trees inequality \cite{van2004detection} (also in \cite[(4.9)]{johnstone}), we have 
    $$\inf_{\htheta_n}\sup_{|\theta|\le (2\epsilon)^{4/3}} \EE\left[(\htheta_n - \theta)^2\right]\ge \frac{1}{2n}\wedge \frac{1}{2}\cdot (2\epsilon)^{8/3}.$$
    Hence we obtain that 
    $$\inf_{\hbtheta_n}\sup_{\btheta\in M}\EE\left[\|\hbtheta_n - \btheta\|^2\right] \ge (2\epsilon)^{-2/3}\cdot \left(\frac{1}{2n}\wedge (2\epsilon)^{8/3}\right) = 2\epsilon^2\wedge \frac{1}{2n\cdot (2\epsilon)^{2/3}}.$$
    Hence in order to achieve no more than $\epsilon^2$ density estimation error (i.e. \pref{eq: density-estimation} holds), we require 
    \begin{equation}\label{eq: density-estimation-lb}
        n_\est(\Gamma, \epsilon)\ge \frac{\epsilon^{-8/3}}{4}.
    \end{equation}
    
    \vspace{0.2cm}
    \textbf{Upper Bound to Density Estimation: }
    Next, we construct an estimator which will achieve $\epsilon^2$ density estimation error (i.e. \pref{eq: density-estimation} holds) with no more than $\tilde{\mathcal{O}}(\epsilon^{-8/3})$ samples. For $n$ samples $\bX = (\bX^{1:n})$ where each $\bX^i\simiid \calN(\btheta, I)$, we consider the following estimator: let $\widehat{\theta}(\bX) = (\widehat{\theta}_1, \widehat{\theta}_2, \cdots) = \frac{1}{n}\sum_{i=1}^n X_i$, and further construct estimator
    $\ttheta(\bX) = (\ttheta_1(\bX), \ttheta_2(\bX), \cdots)$, where 
    $$\ttheta_i = \begin{cases}
        \delta_\lambda (\hat{\theta}_i) &\quad \forall 1\le i\le D,\\
        0 &\quad i > D,
    \end{cases}$$
    where $\lambda > 0$ and $D\in\mathbb{Z}_+$ are parameters to be specified later. Here $\delta_\lambda(\cdot): \mathbb{R}\to \mathbb{R}$ is the soft-thresholding function defined as:
    $$\delta_\lambda(x) = \begin{cases} x - \lambda &\quad x\ge \lambda,\\
    0 &\quad -\lambda \le x < \lambda,\\
    x + \lambda &\quad x < - \lambda.\end{cases}$$
    We will show that with proper choices of $\lambda$ and $D$, the density estimation error can be upper bounded by $\epsilon^2$ with $n = \widetilde{\mathcal{O}}(1/\epsilon^{8/3})$ number of samples. We write $\bX^i = (X_1^i, X_2^i, \cdots)$. Then we have $X_j^1, X_j^2, \cdots X_j^n\simiid \mathcal{N}(\theta_j, 1)$ for any $j\ge 0$. And we only need to verify for some choice of $\lambda$ and $D$, 
    $$\sum_{j=1}^\infty \mathbb{E}\left[(\ttheta_j(\bX) - \theta_j)^2\right]\le \epsilon^2.$$
    When $j > D$, we have $\ttheta_j(\bX) = 0$, hence $$\mathbb{E}\left[(\ttheta_j(\bX) - \theta_j)^2\right] = \theta_j^2.$$
    And for $j \le D$, according to \cite[(8.12)]{johnstone}, we have
    $$\mathbb{E}\left[(\ttheta_j(\bX) - \theta_j)^2\right]\le \frac{1}{n}\exp\left(-\frac{n\lambda^2}{2}\right) + \min\left\{\theta_j^2, \frac{1}{n} + \lambda^2\right\}.$$
    Thus, we obtain
    \begin{align*}
        &\quad \sum_{j=1}^\infty \mathbb{E}\left[(\ttheta_j(\bX) - \theta_j)^2\right]\\
        & \le \frac{D}{n}\exp\left(-\frac{n\lambda^2}{2}\right) + \sum_{j=1}^D \min\left\{\theta_j^2, \frac{1}{n} + \lambda^2\right\} + \sum_{j=D+1}^\infty\theta_j^2.
    \end{align*}
    Since $\btheta\in \Theta$, we have 
    $$\sum_{j=1}^\infty j\cdot |\theta_j|\le 1.$$
    If we choose $D = \lceil 2/\epsilon\rceil$, for $\theta = (\theta_1, \theta_2, \cdots)$, we will have
    \begin{equation}\label{eq: density-estimation-bound-1}
        \sum_{j=D+1}^\infty \theta_j^2\le \frac{1}{D^2}\left(\sum_{j=D+1}^\infty j\cdot |\theta_j|\right)^2\le \frac{1}{D^2}\le \frac{\epsilon^2}{4}.
    \end{equation}

    \par In the next, we will choose the value of $\lambda$ and further upper bound
    \begin{equation}\label{eq:minimizer}\sum_{j=1}^D\min\left\{\theta_j^2, \frac{1}{n} + \lambda^2\right\}.\end{equation}
    Since set $\{(\theta_1, \cdots, \theta_D): \sum_{j=1}^D j\cdot |\theta_j|\le 1\}$ is compact, there must exists some $(\theta_1, \cdots, \theta_D)$ which maximizes \eqref{eq:minimizer}. Without loss of generality we assume $|\theta_j|\le \sqrt{1/n + \lambda^2}$ (otherwise truncate the value of all $\theta_j$ to interval $[-\sqrt{1/n + \lambda^2}, \sqrt{1/n + \lambda^2}]$ and it results in another maximizer). Further if there exists $j_1\neq j_2$ which both satisfy $0 < |\theta_j| < \sqrt{1/n + \lambda^2}$, then by slightly modifying $\theta_{j_1}$ and $\theta_{j_2}$ we can make \eqref{eq:minimizer} larger, while keeping the parameter still in the set. This contradicts to the assumption that $(\theta_1, \cdots, \theta_D)$ is a maximizer. Therefore, there exists at most one of $1\le j\le D$ which satisfies $0 < |\theta_j| < \sqrt{1/n + \lambda^2}$. We let integer $N$ to be the smallest integer such that 
    $$\frac{N(N-1)}{2}\cdot \sqrt{\frac{1}{n} + \lambda^2} \ge 1.$$
    Then the maximizer has at most $N$ nonzero items, which implies
    $$\sup_{\sum_{j=1}^D j\cdot |\theta_j|\le 1}\sum_{j=1}^D\min\left\{\theta_j^2, \frac{1}{n} + \lambda^2\right\}\le N\cdot \left(\frac{1}{n} + \lambda^2\right).$$

    Finally, we choose $\lambda = (\epsilon/8)^{4/3}$. When $n\ge (\epsilon/8)^{-8/3}\cdot 4\log(1/\epsilon)$ we have
    $$N\le 2 \cdot (\epsilon/8)^{-2/3}\quad \text{and}\quad \frac{1}{n}\le \lambda^2,$$
    which implies 
    \begin{equation}\label{eq: density-estimation-bound-2}
        \sum_{j=1}^D\min\left\{\theta_j^2, \frac{1}{n} + \lambda^2\right\}\le 2\cdot (\epsilon/8)^{-2/3}\cdot 2\lambda^2\le \frac{\epsilon^2}{2}.
    \end{equation}
    Further according to our choice of $n$ we have
    \begin{equation}\label{eq: density-estimation-bound-3}
        \frac{D}{n}\exp\left(-\frac{n\lambda^2}{2}\right)\le \frac{\epsilon^2}{2},
    \end{equation}
    Combining \pref{eq: density-estimation-bound-1}, \pref{eq: density-estimation-bound-2} and \pref{eq: density-estimation-bound-3}, we obtain that
    $$\frac{D}{n}\exp\left(-\frac{n\lambda^2}{2}\right) + \sum_{j=1}^D \min\left\{\theta_j^2, \frac{1}{n} + \lambda^2\right\} + \sum_{j=D+1}^\infty\theta_j^2\le \frac{\epsilon^2}{4} + \frac{\epsilon^2}{2} + \frac{\epsilon^2}{4} = \epsilon^2.$$
    Therefore, we obtain 
    \begin{equation}\label{eq: density-estimation-ub}
        n_\est(\Gamma, \epsilon)\le \left(\frac{\epsilon}{8}\right)^{-8/3}\cdot 4\log\left(\frac{1}{\epsilon}\right).
    \end{equation}

    Above all, according to \pref{eq: density-estimation-lb} and \pref{eq: density-estimation-ub}, for set $\Gamma$ defined in \pref{eq: def-theta}, we have 
    $$n_\est(\Gamma, \epsilon) = \tilde{\Theta}\left(\epsilon^{-8/3}\right).$$
\end{proof}

\begin{proof}[Proof of \pref{prop: counter-gof}]
    Our proof is divided into two parts: the lower bound to the goodness-of-fit testing sample complexity and the upper bound to the goodness-of-fit testing sample complexity.

    \vspace{0.2cm}
    \textbf{Lower Bound to Goodness-of-fit Testing: } For any fixed distribution $\mu\in \Delta(\RR^{\infty})$, we define distribution $\PP_{0, X}$ to be the distribution of $(\bX, \bY, \bZ)$ sampled according to the following way: first sample $\btheta\sim \mu$, then sample $\bX = (\bX^{1:n})\simiid \calN(\btheta, I)$. And we define distribution $\PP_{1, X}$ to be the distribution of $\bX = (\bX^{1:n})\simiid \calN(\zero, I)$.
    According to \cite[Lemma 5]{gerber2024likelihood}, we have the following lower bound on the testing error (left hand side of \pref{eq: condition-gof})
    \begin{equation}\label{eq: gof-tv-bound}
        \inf_{\psi}\max_{i\in \{0, 1\}}\sup_{P\in H_i} \PP(\psi(X)\neq i)\ge \frac{1}{2}(1 - \tv(\PP_{0, X}, \PP_{1, X})) - \mu(\Gamma^c) - \mu(B_2(\epsilon)),
    \end{equation}
    where $\Gamma^c$ is the complement of set $\Gamma$, and $B_2(\epsilon)$ denotes the $\ell_2$-ball of radius $\epsilon$. We next construct such a prior $\mu$. First of all, we choose $\mu$ to be supported in the following subset of $\Gamma$:
    $$\Gamma_d = \{\theta = (\theta_1, \theta_2, \cdots): \theta\in \Gamma, \theta_{d+1}  = \theta_{d+2} = \cdots = 0\}\subseteq \Gamma,$$
    where $d$ is some positive integer to be specified later. Since for any $\bX^i\sim \calN(\btheta, I)$, the first $d$ coordinate $X^i_{1:d}$ of $\bX^i$ only depends on $\theta_1, \cdots, \theta_d$, and the rest coordinates are pure i.i.d. standard Gaussian noises, in order to lower bound the goodness-of-fit testing sample complexity of $\Gamma_d$, without loss of generality we only need to consider the first $N$ coordinates, i.e. we can ignore coordinates larger than $N$ and assume $\Gamma_d\subseteq \RR^d$:
    \begin{equation}\label{eq: def-theta-n}
        \Gamma_d = \left\{\theta = (\theta_1, \theta_2, \cdots, \theta_d): \sum_{i=1}^d i\cdot |\theta_i|\le 1\right\}.
    \end{equation}
    
    Next, we construct a distribution $\mu\in \Delta(\Gamma_d)$ with $\Gamma_d$ defined in \pref{eq: def-theta-n}. For some one-dimensional distribution $q\in \Delta([-1, 1])$, in the following form: 
    $$q(\cdot) = (1-h)\delta_0(\cdot) + \frac{h}{2}\delta_r(\cdot) + \frac{h}{2}\delta_{-r}(\cdot),$$
    where $h\in (0, 1)$ and $r\ge 0$ are parameters to be specified later, we consider the following product distribution: 
    \begin{equation}\label{eq: def-nu}
        \mu = \bigotimes_{i=1}^d q\in \Delta(\mathbb{R}^d).
    \end{equation}

    Next, we will lower bound the right hand side of \pref{eq: gof-tv-bound}. First we notice that
    $$\tv(\PP_{0, X}, \PP_{1, X})^2 \le \chi^2(\PP_{1, X}\|\PP_{0, X}) - 1.$$
    We next adopt the Ingster's trick in \cite{ingster1987minimax} and further upper bound the $\chi^2$-divergence: if we use $\varphi_{\btheta}(\cdot)$ to denote the probability density function of $\calN(\btheta, I_D)$, we have
    \begin{align*} 
        \chi^2(\PP_{1, X}\|\PP_{0, X}) & = \EE_{\btheta, \btheta'\simiid \mu} \left[\int_{(\RR^D)^n} \prod_{t=1}^n \frac{\varphi_\btheta(\bz^t)\varphi_{\btheta'}(\bz^t)}{\varphi_{\zero}(\bz^t)} d\bz^1\cdots d\bz^n\right] = \EE_{\btheta, \btheta'\simiid \mu}[ \exp(n\cdot \langle \btheta, \btheta'\rangle)]
    \end{align*}
    Since $\mu$ is the product distribution defined in \pref{eq: def-nu}, we have 
    $$\mathbb{E}_{\btheta, \btheta'\simiid \mu}[\exp(n\cdot \langle \btheta, \btheta'\rangle)] = \prod_{i=1}^d \left(1 + \frac{h^2}{4}\cdot (\exp(nr^2) - 1) + \frac{h^2}{4}\cdot (\exp(-nr^2) - 1)\right).$$
    When $nr^2\le 1$, we have 
    $$\frac{\exp(nr^2) + \exp(-nr^2)}{2} - 1\le n^2r^4.$$
    Therefore, when $dh^2n^2r^4\le 1$, we have 
    \begin{align*} 
        &\hspace{-0.5cm} \prod_{i=1}^d \left(1 + \frac{h^2}{4}\cdot (\exp(nr^2) - 1) + \frac{h^2}{4}\cdot (\exp(-nr^2) - 1)\right)\\
        & \le \prod_{i=1}^d \left(1 + \frac{h^2}{2}\cdot n^2r^2\right)\le 1 + dh^2n^2r^4,
    \end{align*}
    which implies that 
    $$\tv(\PP_{0, X}, \PP_{1, X})\le \sqrt{dh^2n^2r^4}.$$

    According to Hoeffding inequality, for any $\delta > 0$, with probability at least $1 - \delta$ we have 
    $$\sum_{t=1}^d i\cdot |\theta_i| \le d\cdot dhr + d\cdot r\cdot \sqrt{2d\log(1/\delta)}.$$ 
    Additionally, again according to Hoeffding inequality, with probability at least $1 - \delta$ we have 
    $$\sum_{t=1}^d |\theta_t|^2 = \sum_{t=1}^d |\theta_t|^2\ge dhr^2 - r^2\cdot \sqrt{2d\log(2/\delta)}.$$
    As long as $d\ge 12/h^2$, we have with probability at least $4/5$,
    $$\sum_{t=1}^d i\cdot |\theta_i|\le 2d^2hr\quad\text{and}\quad \sum_{t=1}^d |\theta_t|^2\ge \frac{1}{2}dhr^2.$$
    Hence if $2d^2hr\le 1$ and $dhr^2/2\ge \epsilon^2$, we have
    $$\mu(\btheta\in \Gamma^c) + \mu(\btheta\in B_2(\epsilon))\le \frac{1}{5}.$$

    Finally, we choose 
    $$d = \epsilon^{-4/5},\quad h = \frac{1}{16}\epsilon^{2/5}\quad \text{and}\quad r = 8\epsilon^{6/5}.$$
    Then if $n\le \epsilon^{-12/5}/64$ we have $nr^2\le 1$ and also $dh^2n^2r^4 \le 1$. We can also verify that according to the above choices of $(d, h, r)$, $d\ge 12/h^2$, $2d^2hr\le 1$ and $dhr^2/2\ge \epsilon^2$ always holds. Hence have 
    $$\mu(\btheta\in \Gamma^c) + \mu(\btheta\in B_2(\epsilon))\le \frac{1}{5}$$
    and also 
    $$\tv(\PP_{0, X}, \PP_{1, X})\le \sqrt{dh^2n^2r^4}\le \frac{1}{16}.$$
    Bringing these two inequalities back to \pref{eq: gof-tv-bound}, we obtain that if $n\le \epsilon^{-12/5}/64$,
    $$\inf_{\psi}\max_{i\in \{0, 1\}}\sup_{P\in H_i} \PP(\psi(X)\neq i)\ge \frac{1}{2}\cdot \left(1 - \frac{1}{16}\right) - \frac{1}{5} > \frac{1}{4}.$$
    
    \vspace{0.2cm}
    \textbf{Upper Bound to Goodness-of-fit Testing: } Given $n$ samples $\bX = (\bX^{1:n})\simiid \calN(\btheta, I_d)$, we design a test for the testing problem \pref{eq: goodness-of-fit}. We let $D = 2/\epsilon$, and $d = \epsilon^{-4/5}$. For any $1\le i\le D$, we let 
    $$\htheta_i = \frac{1}{n}\sum_{t=1}^n X^t_i,\quad \text{where } X^t_i\text{ is the }i\text{-th coordinates of }\bX^t.$$
    Consider the following test: 
    \begin{enumerate}[label=(\roman*)]
        \item If $\sum_{i=1}^d (\htheta_i)^2\le \epsilon^2/2$ and $\max_{d\le i\le D} |\htheta_i|\le \epsilon$, we accept the null hypothesis $\btheta = \zero$;
        \item If $\sum_{i=1}^d (\htheta_i)^2 > \epsilon^2/2$ or $\max_{d\le i\le D} |\htheta_i| > \epsilon$, we reject the null hypothesis $\btheta = \zero$.
    \end{enumerate}
    Suppose the above testing scheme is $\psi$. Next we will verify \pref{eq: condition-gof} for the testing scheme $\psi$ as long as the number of sample satisfies
    \begin{equation}\label{eq: condition-n}
        n\ge \epsilon^{-12/5}\log\left(\frac{16}{\epsilon}\right).
    \end{equation}
    First we verify the case where $i = 0$, i.e.
    \begin{equation}\label{eq: gof-i=0}
        \sup_{P\in H_0} \PP(\psi(\bX) = 1)\le \frac{1}{4}.
    \end{equation}
    Notice that when $\btheta = 0$, we have 
    $$\htheta_i\simiid \calN(0, 1/n)\quad \forall 1\le i\le D.$$
    Hence we only need to prove that 
    \begin{equation} \label{eq: upper-gof-first-case}
        \PP\left(\sum_{i=1}^d (\htheta_i)^2 > \frac{\epsilon^2}{2}\right)\le \frac{1}{8}\quad \text{and}\quad \PP\left(\max_{d+1\le i\le D} |\htheta_i|\le \epsilon\right)\ge \frac{7}{8}.
    \end{equation}
    To verify the first inequality of \pref{eq: upper-gof-first-case}, we calculate 
    $$\EE\left[\sum_{i=1}^d (\htheta_i)^2\right] = \frac{d}{n}\quad \text{and}\quad \var\left[\sum_{i=1}^d (\htheta_i)^2\right] = \frac{2d}{n^2}.$$
    Hence according to Chebyshev inequality we obtain that when $n$ satisfies \pref{eq: condition-n}, we have
    $$\PP\left(\sum_{i=1}^d (\htheta_i)^2 > \frac{\epsilon^2}{8}\right)\le \frac{d/n^2}{|\epsilon^2/2 - d/n|^2}\le \frac{1}{8}.$$
    Additionally, since $\htheta_i\simiid \calN(0, 1/n)$, when $n$ satisfies \pref{eq: condition-n}, we have
    $$\PP\left(\max_{d+1\le i\le D} |\htheta_i|\le \epsilon^{6/5}\right) = \left(1 - \varphi_1(\epsilon^{6/5}\sqrt{n})\right)^{D-d}\ge 1 - D\cdot \varphi_1(\epsilon^{6/5}\sqrt{n})\ge 1 - D\cdot \exp\left(- n\epsilon^{12/5}\right)\ge \frac{7}{8}.$$
    where $\varphi_1(\cdot)$ is the probability density function of one-dimensional standard normal distribution. Hence both inequalities in \pref{eq: upper-gof-first-case} are verified.

    Next, we verify \pref{eq: condition-gof} for the case $i = 1$, i.e. 
    \begin{equation}\label{eq: gof-i=1}
        \sup_{P\in H_1} \PP(\psi(\bX) = 0)\le \frac{1}{4}.
    \end{equation}
    We only need to show that for any $\btheta\in \Gamma$ with $\|\btheta\|_2\ge \epsilon$, we always have 
    $$\EE_{\bX\simiid \calN(\btheta, I)}\left[ \psi(\bX) = 1\right]\le \frac{1}{4}.$$
    Notice that when $\bX\simiid \calN(\btheta, I)$, we have $\htheta_i \sim \calN(\theta_i, 1/n)$ independently. For $\btheta\in \Gamma$ which satisfies $\|\btheta\|_2\ge \epsilon$, we have 
    $$\sum_{i=D+1}^\infty (\theta_i)^2 \le \left(\sum_{i=D+1}^\infty |\theta_i|\right)\le \frac{1}{D^2}\left(\sum_{i=D+1}^\infty i\cdot |\theta_i|\right)\le \frac{\epsilon^2}{4},$$
    which implies that 
    $$\sum_{i=1}^d (\theta_i)^2 + \sum_{i=d+1}^D (\theta_i)^2 = \sum_{i=1}^D (\theta_i)^2 \ge \frac{\epsilon^2}{2}.$$
    Therefore, we have either $\sum_{i=1}^d (\theta_i)^2\ge \epsilon^2/4$, or $\sum_{i=d+1}^D (\theta_i)^2\ge \epsilon^2/4$. If we have $\sum_{i=1}^d (\theta_i)^2\ge \epsilon^2/4$, since $\htheta_i\sim \calN(\theta_i, 1/n)$ independently, we can calculate 
    \begin{align*} 
        \EE\left[\sum_{i=1}^d (\htheta_i)^2\right] & = \sum_{i=1}^d (\theta_i)^2 + \frac{d}{n}\\
        \var\left[\sum_{i=1}^d (\htheta_i)^2\right] & = \frac{4}{n}\cdot \sum_{i=1}^d (\theta_i)^2 + \frac{2d}{n^2}.
    \end{align*}
    Therefore, according to Chebyshev inequality we obtain that when $n$ satisfies \pref{eq: condition-n}
    $$\PP\left(\sum_{i=1}^d (\htheta_i)^2 
    \le \frac{\epsilon^2}{8}\right)\le \frac{\frac{4}{n}\cdot \sum_{i=1}^d (\theta_i)^2 + \frac{2d}{n^2}}{|\sum_{i=1}^d (\theta_i)^2 + \frac{d}{n} - \frac{\epsilon^2}{8}|^2}\le \frac{1}{4},\quad \text{hence}\quad \PP\left(\sum_{i=1}^d (\htheta_i)^2 
    \ge \frac{\epsilon^2}{8}\right)\ge \frac{3}{4}.$$
    Next, we consider the case where $\sum_{i=d+1}^D (\theta_i)^2\ge \epsilon^2/4$. We notice that
    $$\sum_{i=d+1}^D (\theta_i)^2\le \left(\sum_{i=d+1}^D |\theta_i|\right)\cdot \max_{d+1\le i\le D} |\theta_i|\le \frac{1}{d}\cdot \left(\sum_{i=d+1}^D i\cdot |\theta_i|\right)\cdot \max_{d+1\le i\le D} |\theta_i|\le \frac{1}{d}\cdot \max_{d+1\le i\le D}|\theta_i|,$$
    which implies that 
    $$\max_{d+1\le i\le D}|\theta_i|\ge d\cdot \frac{\epsilon^2}{4}\ge \epsilon^{6/5}.$$ 
    Since $\htheta_i\sim \calN(\theta_i, 1/n)$ independently, when $n$ satisfies \pref{eq: condition-n}, we have 
    \begin{align*} 
        \PP\left(\max_{d+1\le i\le D} |\htheta_i - \theta_i|\le \epsilon^{6/5}\right) & = (1 - \varphi_1(\epsilon^{6/5}\sqrt{n}))^{D - d}\\
        & \ge 1 - D\cdot \varphi(\epsilon^{6/5}\sqrt{n}) \ge 1 - D\cdot \exp(-n\epsilon^{12/5})\ge \frac{3}{4},
    \end{align*}
    which implies that 
    $$\PP\left(\max_{d+1\le i\le D} |\htheta_i|\ge \epsilon^{6/5}\right)\ge \frac{3}{4}.$$
    According to the form of the testing scheme $\psi$, we verified \pref{eq: condition-gof} for the case $i = 1$.
\end{proof}

\subsection{Proofs of \pref{thm: gof-est} and \pref{corr: gof-est}}\label{sec: app-gof-est-lower}
In order to prove \pref{thm: gof-est}, we first bring up the following property of orthosymmetric convex sets.
\begin{proposition}\label{prop: convex-symmetric}
    Suppose $\Gamma\subseteq \RR^d$ is an orthosymmetric set. If $\Gamma$ is also convex, then for any $\btheta = (\theta_1, \cdots, \theta_D)\in \Gamma$ and $\alpha = (\alpha_1, \cdots, \balpha_D)\in [-1, 1]^D$, we have $\btheta\cdot \balpha = (\alpha_1\theta_1, \alpha_2\theta_2, \cdots, \alpha_D\theta_D)\in \Gamma$.
\end{proposition}
Now we are ready for the proof of \pref{thm: gof-est}.
\begin{proof}[Proof of \pref{thm: gof-est}]
    Without loss of generality, we assume that $n_\est(\Gamma, \epsilon)\ge 4$. Our proof proceeds as follows. First, we consider parameter estimation via soft-thresholding. We show that there must exist $\btheta^*$ in $\Gamma$ such that $\|\btheta^*\|^2 \asymp \epsilon^2$ and each entry $|\btheta^*_i| \le n_\est(\Gamma, \epsilon)^{-1/2} \mathrm{polylog}(n,\epsilon)$ (for otherwise, soft-thresholding estimator would beat the optimal sample complexity of estimation). Second, we use Ingster's method of simple ($\btheta=0$) vs composite ($\btheta = (\pm \btheta^*_1,\ldots,\pm \btheta^*_D)$) hypothesis testing to lower bound the goodness of fit sample complexity $n_\gof(\Gamma, \epsilon)$. 

    \vspace{0.2cm}
    \textbf{First Part: Density Estimation Rate: } We define the soft-thresholding function $\stf: \RR\times (\RR_+\cup\{0\})\to \RR$:
    \begin{equation}\label{eq: def-stf}
        \stf(x, \lambda) = \begin{cases}x - \lambda &\quad \text{if }x\ge \lambda,\\0 &\quad \text{if }-\lambda\le x < \lambda,\\ x + \lambda & \quad \text{if }x < -\lambda.\end{cases}
    \end{equation}
    For some $\btheta\in \Gamma$, given $n$ samples $X = X^{1:n}\sim \mathcal{N}(\btheta, I_D)^{\otimes n}$, we construct the soft-thresholding estimator $\hbtheta_\stf^n = (\htheta_1^n, \cdots, \htheta_D^n)$ as follows: for any $j\in [D]$, we choose
    $$\htheta_j^n = \stf\left(\frac{1}{n}\sum_{i=1}^nX_j^i, \lambda(n)\right),$$
    where we denote $X^i = (X_1^i, X_2^i, \cdots, X_D^i)$, and we let
    \begin{equation}\label{eq:def-lambda}\lambda(n)\triangleq \sqrt{\frac{2\log\left(\nicefrac{2D}{(n\epsilon^2)}\right)\vee 0}{n}}.\end{equation}

    We let $n_\stf(\epsilon)$ to denote the smallest $n$ such that $\hbtheta_\stf^n$ induces expected estimation error $\epsilon^2$ for all $\btheta\in\Gamma$, i.e.
    \begin{equation}\label{eq: property-n-stf}
        n_\stf(\epsilon) = \min\left\{n: \sup_{\btheta\in \Gamma}\EE\left[\|\hbtheta_\stf^n - \btheta\|^2\right]\le \epsilon^2\right\}.
    \end{equation}
    Since $n_\est(\Gamma, \epsilon)$ is the minimal number of samples in order to reach expected estimation error $\epsilon^2$, we have
    $$n_\stf(\epsilon)\ge n_\est(\Gamma, \epsilon).$$
    According to \pref{eq: property-n-stf}, we have
    \begin{equation}\label{eq:stf-upper}\sup_{\btheta\in \Gamma}\mathbb{E}_{X_{1:n_\stf(\epsilon)}\stackrel{\iid}{\sim} \mathcal{N}(\btheta, I_D)}\left[\left\|\hbtheta_\stf^{n_\stf(\epsilon)} - \btheta\right\|_2^2\right]\le \epsilon^2\end{equation}
    and for any $n < n_\stf(\epsilon)$, 
    \begin{equation}\label{eq:eq1}\sup_{\btheta\in \Gamma}\mathbb{E}_{X_{1:n}\stackrel{\iid}{\sim}  \mathcal{N}(\btheta, I_D)}\left[\left\|\hbtheta_\stf^n - \btheta\right\|_2^2\right]\ge \epsilon^2.\end{equation}

    Next, according to \cite[(8.7), (8.12)]{johnstone}, we have for any positive integer $n$ and $\btheta = (\theta_1, \cdots, \theta_D)\in \Gamma$ that
    $$\mathbb{E}_{X_{1:n}\stackrel{\iid}{\sim} \mathcal{N}(\btheta, I_D)} \left[\left\|\hbtheta_\stf^n - \btheta\right\|_2^2\right]\le \sum_{i=1}^D\frac{1}{n}\exp\left(-\frac{n\lambda(n)^2}{2}\right) + \min\left\{\theta_i^2, \frac{1}{n} + \lambda(n)^2\right\},$$
    which implies
    \begin{align*}
        \sup_{\btheta\in\Gamma}\mathbb{E}_{X_{1:n}\sim \mathcal{N}(\btheta, I_D)} \left[\left\|\hbtheta_\stf^n - \btheta\right\|_2^2\right] \le \sup_{\btheta\in\Gamma} \left\{\sum_{i=1}^D\frac{1}{n}\exp\left(-\frac{n\lambda(n)^2}{2}\right) + \min\left\{\theta_i^2, \frac{1}{n} + \lambda(n)^2\right\}\right\}.
    \end{align*}
    With our choice of $\lambda(n)$ in \eqref{eq:def-lambda}, we have for any $1\le i\le D$,
    $$\sum_{i=1}^D\frac{1}{n}\exp\left(-\frac{n\lambda(n)^2}{2}\right) \le \frac{\epsilon^2}{2}\quad \text{and}\quad \frac{1}{n} + \lambda(n)^2\le \frac{1 + (2\log\left(\nicefrac{2D}{(n\epsilon^2)}\right)\vee 0)}{n},$$
    which implies
    $$\sup_{\btheta\in\Gamma}\mathbb{E}_{X_{1:n}\sim \mathcal{N}(\btheta, I_D)} \left[\left\|\hbtheta_\stf - \btheta\right\|_2^2\right]\le \frac{\epsilon^2}{2} + \sup_{\btheta\in\Gamma}\sum_{i=1}^D \min\left\{\theta_i^2, \frac{1 + \log\left(\nicefrac{2D}{(n\epsilon^2)}\right)}{n}\right\}.$$
    We define set $L\subseteq \RR^D$ as
    $$L = \left\{\btheta = (\theta_1, \cdots, \theta_D)\ \Big{|}\ |\theta_i|\le \sqrt{\frac{2\log(\nicefrac{2D}{(n\epsilon^2)})\vee 0 + 1}{n}}\right\}$$
    Note that according to \pref{prop: convex-symmetric} we can replace $\sup_{\btheta\in \Gamma}$ with $\sup_{\btheta \in \Gamma \cap L}$ obtaining
    \begin{align*} 
        &\hspace{-0.5cm} \sup_{\btheta\in\Gamma}\sum_{i=1}^D \min\left\{\theta_i^2, \frac{2\log\left(\nicefrac{2D}{(n\epsilon^2)}\right)\vee 0 + 1}{n}\right\}\\
        & = \sup_{\btheta\in\Gamma\cap L}\sum_{i=1}^D \min\left\{\theta_i^2, \frac{2\log\left(\nicefrac{2D}{(n\epsilon^2)}\right)\vee 0 + 1}{n}\right\} = \sup_{\btheta \in \Gamma \cap L} \|\btheta\|^2\,.
    \end{align*}
    
    We choose $n = n_\stf(\epsilon) - 1$. Since $\Gamma \cap L$ is compact, the above supreme is achieved at some $\btheta^* = (\theta_1^*, \cdots, \theta_D^*) \in \Gamma \cap L$, which has two properties. On one hand it satisfies
    \begin{equation}\label{eq:theta-star}
        |\theta_i^*|\le \sqrt{\frac{2\log\left(\nicefrac{2D}{((n_\est(\epsilon) - 1)\epsilon^2)}\right)\vee 0 + 1}{n_\stf(\epsilon) - 1}},\end{equation}
    On the other hand, according to \eqref{eq:eq1}, we obtain that
    \begin{align*}
        \sum_{i=1}^D(\theta_i^*)^2 & \ge \sup_{\btheta\in\Gamma}\mathbb{E}_{X_{1:n_\stf(\epsilon)-1}\stackrel{\iid}{\sim} \mathcal{N}(\btheta, I_D)} \left[\left\|\hbtheta_\stf^{n_\stf(\epsilon)-1} - \btheta\right\|_2^2\right] - \frac{\epsilon^2}{2} \ge \frac{\epsilon^2}{2}.
    \end{align*}

    Additionally, according to~\cite[Lemma 8.3]{johnstone}, we have for any positive integer $n$ and $\btheta = (\theta_1, \cdots, \theta_D)\in\Gamma$,
    $$\mathbb{E}_{X_{1:n}\stackrel{\iid}{\sim} \mathcal{N}(\btheta, I_D)} \left[\left\|\hbtheta_\stf^n - \btheta\right\|_2^2\right]\ge \frac{1}{2}\cdot \sum_{i=1}^D\min\left\{\theta_i^2, \frac{1}{n} + \lambda(n)^2\right\}$$
    We choose $n = n_\stf(\epsilon)$ and $\btheta = \btheta^* = (\theta_1^*, \cdots, \theta_D^*)$ defined above. \eqref{eq:stf-upper} implies
    \begin{align*}
        \sum_{i=1}^D\min\left\{(\theta_i^*)^2, \frac{1}{n_\stf(\epsilon)} + \lambda(n_\stf(\epsilon))^2\right\}\le 2\sup_{\btheta\in\Gamma}\mathbb{E}_{X_{1:n_\stf(\epsilon)}\stackrel{\iid}{\sim} \mathcal{N}(\btheta, I_D)} \left[\left\|\hbtheta_\stf^{n_\stf(\epsilon)} - \btheta\right\|_2^2\right]\le 2\epsilon^2.
    \end{align*}
    Further when $n_\stf(\epsilon)\ge 4$, \pref{eq:theta-star} gives that for any $i\in [D]$,
    \begin{align*}
        \frac{1}{n_\stf(\epsilon)} + \lambda(n_\stf(\epsilon))^2 & = \frac{1}{n_\stf(\epsilon)} + \frac{2\log\left(\nicefrac{2D}{(n_\stf(\epsilon)\epsilon^2)}\right)\vee 0 + 1}{n_\stf(\epsilon)}\\
        & \ge \frac{1}{2}\cdot \frac{2\log\left(\nicefrac{2D}{((n_\stf(\epsilon) - 1))\epsilon^2})\right)\vee 0 + 1}{n_\stf(\epsilon)-1} \ge \frac{1}{4}(\theta_i^*)^2,
    \end{align*}
    Hence we obtain that 
    $$\sum_{i=1}^D (\theta_i^*)^2\le 4\sum_{i=1}^D\min\left\{(\theta_i^*)^2, \frac{1}{n_\stf(\epsilon)} + \lambda(n_\stf(\epsilon))^2\right\}\le 8\epsilon^2.$$
    
    Overall, we constructed $\btheta^* = (\theta_1^*, \cdots, \theta_D^*)\in\Gamma$ such that
    \begin{enumerate}[label=(\alph*)]
        \item For any $1\le i\le D$, $|\theta_i^*|\le \sqrt{\frac{2\log\left(\nicefrac{2D}{((n_\est(\epsilon) - 1)\epsilon^2)}\right)\vee 0 + 1}{n_\stf(\epsilon) - 1}}$. \label{theta-star-a}
        \item $\sum_{i=1}^D(\theta_i^*)^2 \ge \nicefrac{\epsilon^2}{2}$. \label{theta-star-b}
        \item $\sum_{i=1}^D(\theta_i^*)^2 \le 8\epsilon^2$. \label{theta-star-c}
    \end{enumerate}

    \vspace{0.2cm}
    \textbf{Second Part: Goodness of Fit Rate: } Suppose we already identified $\btheta^*$ which satisfies \ref{theta-star-a}, \ref{theta-star-b} and \ref{theta-star-c}. We consider distribution $\nu = \nu_1\otimes \nu_2\otimes\cdots\otimes \nu_D$ where $\nu_i\in\Delta(\mathbb{R})$ is given by
    $$\nu_i(\cdot) = \frac{\delta_{\theta_i^*}(\cdot)}{2} + \frac{\delta_{-\theta_i^*}(\cdot)}{2}.$$
    Since $\btheta^*\in\Gamma$ and $\Gamma$ is orthosymmetric, we have $\nu\in\Delta(\Gamma)$. And \ref{theta-star-b} gives for any $\btheta\in\mathrm{supp}(\nu)$, we always have $\|\btheta\|\ge \nicefrac{\epsilon}{\sqrt{2}}$. Thus, according to \cite[Proposition 2.11]{ingster2012nonparametric}, for any testing scheme $\psi$ with $n$ samples for the following testing problem:
    $$H_0: \btheta = \zero\quad \text{versus}\quad H_1: \|\btheta\|_2\ge \frac{\epsilon}{\sqrt{2}},\  \btheta \in \Gamma$$
    as long as $n\le \sqrt{\frac{n_\est(\Gamma, \epsilon)}{64\epsilon^2\log(2D)}}$, we have
    $$\sup_{P\in H_0} \PP(\psi(X) \neq i) + \sup_{P\in H_1} \PP(\psi(X) \neq i)\ge 1 - \frac{1}{2}\tv(\mathbb{E}_{\btheta\sim \nu}[P_{\btheta}^{\otimes n}], P_0^{\otimes n})\,,$$
    where we denoted $P_\btheta = \mathcal{N}(\btheta, I_D)$.
    Hence, if we can show that the $\TV \ge \frac{1}{2}$ for a certain value of $n$ it must imply that $n_\gof(\Gamma, \nicefrac\epsilon{\sqrt{2}}) \ge n$. We will indeed show that $n=\sqrt{\frac{n_\est(\Gamma, \epsilon)}{64\epsilon^2\log(2D)}}$ works.

    To that end, we first use a standard bound  \cite[Proposition 7.15]{polyanskiy2025information}
    $$\tv(\mathbb{E}_{P\sim \mu(\nu)}[P^{\otimes n}], P_0^{\otimes n})\le \sqrt{\frac{1}{4}\chi^2\left(\mathbb{E}_{P\sim \mu(\nu)}[P^{\otimes n}]\|P_0^{\otimes n}\right)}\,.$$
    Next, we have the standard bound of Ingster 
    \begin{align*}
        &\hspace{-1cm} \chi^2\left(\mathbb{E}_{P\sim \mu(\nu)}[P^{\otimes n}]\|P_0^{\otimes n}\right) \stackrel{(i)}{\le} \prod_{i=1}^D\exp\left(\frac{1}{2}n^2(\theta_i^*)^4\right)-1 = \exp\left(\frac{n^2}{2}\sum_{i=1}^D (\theta_i^*)^4\right)-1,
    \end{align*}
    where $(i)$ uses \cite[(3.68)]{ingster1987minimax} 
    and the inequality $\frac{1}{2}\exp(x) + \frac{1}{2}\exp(-x)\le \exp\left(\frac{1}{2}x^2\right)$. Next according to \ref{theta-star-a} and \ref{theta-star-c}, we obtain that when $n_\est(\Gamma, \epsilon)\ge 4$,
    \begin{align*}
        \sum_{i=1}^D (\theta_i^*)^4 & \le \left(\max_{i\in [D]} |\theta_i^*|\right)^2\cdot \left(\sum_{i=1}^D (\theta_i^*)^2\right)\\
        & \le 8\epsilon^2\cdot \frac{2\log\left(\nicefrac{2D}{((n_\est(\epsilon) - 1)\epsilon^2)}\right)\vee 0 + 1}{n_\stf(\epsilon) - 1} \le 64\epsilon^2\cdot \frac{\log\left(\nicefrac{2D}{(n_\est(\Gamma, \epsilon)\epsilon^2)}\right)\vee 0 + 1}{n_\est(\Gamma, \epsilon)}.
    \end{align*}
    Next, we notice that when the diameter $\diam(\Gamma)$ of $\Gamma$ is less than $\epsilon$, then the density estimation rate $n_\est(\Gamma, \epsilon)$ is zero, hence the inequality in \pref{thm: gof-est} obviously holds. When $\diam(\Gamma)\ge \epsilon$, the density estimation complexity $n_\est(\Gamma, \epsilon)$ satisfies $n_\est(\Gamma, \epsilon)\ge \epsilon^{-2}$. Hence we obtain that
    $$\log\left(\nicefrac{2D}{(n_\est(\Gamma, \epsilon)\epsilon^2)}\right)\vee 0 + 1\le \log(2D)$$
    Hence when $n$ satisfies
    $$n^2\le \frac{n_\est(\Gamma, \epsilon)}{64\epsilon^2\cdot \log(2D)},$$
    we have
    $$\chi^2\left(\mathbb{E}_{P\sim \mu(\nu)}[P^{\otimes n}]\|P_0^{\otimes n}\right)\le \exp(1/2) - 1\le \frac{2}{3},$$
    completing the proof.
\end{proof}

\begin{proof}[Proof of \pref{corr: gof-est}]
    We fix $\epsilon > 0$ and let $D = D_\coor(\Gamma, \epsilon)$. According to the definition of coordinate-wise Kolmogorov dimension \pref{def: kolmogorov-coor}, there exists a subset $A$ of $\mathbb{Z}_+$ such that 
    $$\sup_{\btheta\in \Gamma} \sum_{i\in \mathbb{Z}_+\backslash A}(\theta_i)^2\le 2\epsilon^2.$$
   Without loss of generality that we assume $A = \{1, 2, \cdots, D\}$, and we use $\Pi_D$ to denote the projection operator onto the first $D$ coordinates.

    We construct $\Gamma_D = \{\Pi_D \btheta: \btheta\in \Gamma\}$. Then for any density estimator $\hbtheta_D$ for $\Gamma_D$, we can construct a density estimator $\hbtheta$: when taking $X$, we define estimator $\hbtheta$ as
    $$\hbtheta(X) = \big(\underbrace{\htheta_1(X), \htheta_2(X), \cdots, \htheta_k(X)}_{\text{first } D \text{ coordinates match }\hbtheta_D(\Pi_D X)}, 0, 0, \cdots\big),$$
    where $\Pi_D X$ denotes the data after taking projection $\Pi_D$ for all elements in $X$. Then we have for any $\btheta\in \Gamma$, 
    $$\EE\left[\|\btheta - \hbtheta(X)\|_2^2\right] = \EE\left[\|\Pi_D \btheta - \hbtheta_D(\Pi_D X)\|_2^2\right] + \|\Pi_D \btheta - \btheta\|^2\le \EE\left[\|\Pi_D \btheta - \hbtheta_D(\Pi_D X)\|_2^2\right] + 2\epsilon^2.$$
    Notice that $\Pi_D X$ can be viewed as i.i.d. samples collected through $\calN(\Pi_D\btheta, I_D)$, we obtain that 
    \begin{equation}\label{eq: kolmogorov-eq1}
        n_\est(\Gamma, \sqrt{3}\epsilon)\le n_\est(\Gamma_D, \epsilon).
    \end{equation}

    Additionally, since $\Gamma$ is orthosymmetric and convex, we have $\Gamma_D\subseteq \Gamma$. When carrying the goodness of fit test for $\Gamma_D$, the data after $D$-th coordinate are generated according to $\calN(0, 1)$ and independent to the parameter $\btheta\in \Gamma_D$. Therefore, carrying goodness of fit testing to $\Gamma$ is no easier than carrying goodness of fit testing to $\Gamma_D$, which implies 
    \begin{equation}\label{eq: kolmogorov-eq2}
        n_\gof\left(\Gamma, \frac{\epsilon}{\sqrt{2}}\right)\ge n_\gof\left(\Gamma_D, \frac{\epsilon}{\sqrt{2}}\right).
    \end{equation}

    Since $\Gamma$ is compact, orthosymmetric and convex, $\Gamma_D$ is a $D$-dimensional compact, orthosymmetric and convex set as well. Hence \pref{thm: gof-est} implies that
    $$n_\gof\left(\Gamma_D, \frac{\epsilon}{\sqrt{2}}\right)^2\ge \frac{1}{64\log(2D)}\cdot \frac{n_\est(\Gamma_D, \epsilon)}{\epsilon^2}.$$
    Bringing in \pref{eq: kolmogorov-eq1} and \pref{eq: kolmogorov-eq2}, we obtain that
    \begin{align*} 
        n_\gof\left(\Gamma, \frac{\epsilon}{\sqrt{2}}\right)^2 & \ge \frac{1}{64\log(2D)}\cdot \frac{n_\est(\Gamma, \sqrt{3}\epsilon)}{\epsilon^2}\\
        & = \frac{1}{64\log(2D)}\cdot \frac{n_\est(\Gamma, \sqrt{3}\epsilon)}{\epsilon^2}.
    \end{align*}
\end{proof}

\subsection{Discussion between Kolmogorov Dimension and Coordinate-wise Kolmogorov Dimension}\label{sec: coor-kolmogorov}

For any orthosymmetric set $\Gamma$, we have the following relationship, which provides an upper bound of the coordinate-wise Kolmogorov dimension $D_\coor(\Gamma, \epsilon)$ in terms of the traditional Kolmogorov dimension $D(\Gamma, \epsilon)$.

\begin{proposition}\label{prop: kol-coor-kol}
    For any orthosymmetric, convex, compact set $\Gamma$, we have 
    $$D\left(\Gamma, \frac{\epsilon}{D_\coor(\Gamma, \epsilon)}\right)\ge \frac{D_\coor(\Gamma, \epsilon)}{2} - 2.$$
\end{proposition}
The proof of \pref{prop: kol-coor-kol} requires the following lemmas.
\begin{lemma}\label{lem: orthogonal-vector}
    Suppose set $\Lambda = \{\be_1, \cdots, \be_d\}$ consists of orthogonal unit vectors. Then for any $(d-1)$-dimensional projection $\Pi$, we have 
    \begin{equation}\label{eq: orthogonal-vector}
        \max_{i\in [d]} \left\|\be_i - \Pi[\be_i]\right\|\ge \frac{1}{\sqrt{d}}.
    \end{equation}
\end{lemma}
\begin{proof}[Proof of \pref{lem: orthogonal-vector}]
    Suppose $\bv_1, \cdots, \bv_{d-1}$ to be a set of orthogonal basis of the image of projection $\Pi$. Then for any $i\in [d]$,
    $$\left\|\be_i - \Pi[\be_i]\right\|^2 = \|\be_i\|^2 - \|\Pi[\be_i]\|^2 = 1 - \sum_{j=1}^{d-1} \langle \be_i, \bv_j\rangle^2,$$
    which implies that 
    $$\sum_{i=1}^d \left\|\be_i - \Pi[\be_i]\right\|^2 = d - \sum_{i=1}^d\sum_{j=1}^{d-1} \langle \be_i, \bv_j\rangle^2 = d - \sum_{j=1}^{d-1}\left(\sum_{i=1}^d \langle \be_i, \bv_j\rangle^2\right).$$
    Since $\be_1, \cdots, \be_d$ are orthogonal unit vectors,
    $$\sum_{i=1}^d \langle \be_i, \bv_j\rangle^2 = \left\|\Pi_{\mathrm{span}(\be_1, \cdots, \be_d)} [\bv_j]\right\|\le 1,\qquad \forall j\in [d-1],$$
    where the above $\Pi$ denotes the projection onto the space spanned by $\be_1, \cdots, \be_d$. Therefore,
    $$\sum_{i=1}^d \left\|\be_i - \Pi[\be_i]\right\|\ge d - (d-1) = 1,$$
    which implies \pref{eq: orthogonal-vector}.
\end{proof}
\begin{lemma}\label{lem: coor-kol}
    For orthosymmetric, convex, compact set $\Gamma$, suppose $d = D_\coor(\Gamma, \epsilon)-1$, then there exists $d$ orthogonal vectors $\bu_1, \cdots, \bu_{\lfloor d/2\rfloor}\in \Gamma$ such that for any $1\le i\le \lfloor d/2\rfloor$, $\|\bu_i\|\ge \epsilon/\sqrt{d}$.
\end{lemma}
\begin{proof}[Proof of \pref{lem: coor-kol}]
    Suppose $\be_1, \be_2, \cdots$ are unit vectors of each coordinates, and we let $v_i = \sup\{v\ge 0: v\be_i\in \Gamma\}$. Then since $\Gamma$ is closed, we have $\bv_i = v_i\be_i\in \Gamma$. Without loss of generality we assume $v_1\ge v_2\ge \cdots$. Then according to the definition of coordinate-wise Kolmogorov dimension in \pref{def: kolmogorov-coor}, we have 
    $$\sum_{i=d+1}^D (v_i)^2 > \epsilon^2.$$
    If $v_d\ge \epsilon / \sqrt{d}$, then by choosing $\bu_i = \bv_i$, these vectors satisfy the conditions. Next we assume $v_d < \epsilon / \sqrt{d}$. To construct $\bu_1, \cdots, \bu_{\lfloor d/2\rfloor}$, we initiate the following process: letting $\bu_1 = \sum_{i=d+1}^{k_1} (v_i)^2$, where $k_1$ is the smallest number such that $\|\bu_1\|\ge \epsilon / \sqrt{d}$. Then let $\bu_2 = \sum_{i=k_1+1}^{k_2} (v_i)^2$ where $k_2$ is the smallest number such that $\|\bu_2\|\ge \epsilon / \sqrt{d}$, and so on. Since $v_i\le v_d < \epsilon / \sqrt{d}$ holds for any $i\ge d+1$, the above construction gives that
    $$\frac{\epsilon}{\sqrt{d}}\le \|\bu_i\|\le \frac{\sqrt{2}\epsilon}{\sqrt{d}}.$$
    Therefore, since $\sum_{i=d+1}^D (v_i)^2 > \epsilon^2$, the above process can proceed at least $\lfloor d/2\rfloor$ times. Hence it is eligible to construct $\bu_1, \cdots, \bu_{\lfloor d/2\rfloor}$ so that $\|\bu_i\|\ge \epsilon/\sqrt{d}$ holds. It is easy to verify that $\bu_1, \cdots, \bu_{\lfloor d/2\rfloor}$ are orthogonal. And since $\Gamma$ is orthosymmetric and convex, we have $\bu_i\in \Gamma$ for any $i$.
\end{proof}
Now we are ready to prove \pref{prop: kol-coor-kol}.
\begin{proof}[Proof of \pref{prop: kol-coor-kol}]
    Let $d = D_\coor(\Gamma, \epsilon) - 1$. According to \pref{lem: coor-kol}, there exists orthogonal vectors $\bu_1, \cdots, \bu_{\lfloor d/2\rfloor}\in \Gamma$ such that $\|\bu_i\|\ge \epsilon / \sqrt{d}$. By injecting orthogonal unit vectors $\sqrt{d}\bu_i / \epsilon$ into \pref{lem: orthogonal-vector}, we have 
    $$\inf_{\Pi}\max_{i\in \lfloor d/2\rfloor}\left\|\bu_i - \Pi[\bu_i]\right\|\ge \frac{\epsilon}{\sqrt{d}\cdot \sqrt{\lfloor d/2\rfloor}}> \frac{\epsilon}{D_\coor(\Gamma, \epsilon)},$$
    where the infimum is over all possible $\lfloor d/2\rfloor - 1$ projections. Since $\bu_i\in \Gamma$, we obtain that 
    $$\inf_{\Pi}\max_{\bu\in \Gamma}\left\|\bu - \Pi[\bu]\right\| > \frac{\epsilon}{D_\coor(\Gamma, \epsilon)}.$$
    Hence we have the following lower bound to the traditional Kolmogorov dimension:
    $$D\left(\Gamma, \frac{\epsilon}{D_\coor(\Gamma, \epsilon)}\right)\ge \lfloor d/2\rfloor - 1 \ge \frac{D_\coor(\Gamma, \epsilon)}{2} - 2.$$
\end{proof}

\pref{prop: kol-coor-kol} has the following direct corollary.
\begin{corollary}
    Suppose orthosymmetric, convex, compact set $\Gamma$ satisfies $D(\Gamma, \epsilon)\lesssim \epsilon^{-p}$ for some $0 < p < 1$. Then the coordinate-wise Kolmogorov dimension satisfies 
    $$D_\coor(\Gamma, \epsilon)\lesssim \epsilon^{-p/(1-p)}.$$
\end{corollary}

\subsection{Alternative proof of \pref{prop: gof-est-upper}}\label{sec: app-gof-est-quadratic}
We present a version of the second inequality in \pref{prop: gof-est-upper} in this section.
\begin{proposition}\label{prop: gof-est-upper-1}
    Suppose $\Gamma$ any set such that the projection estimator is minimax optimal up to constants, i.e. there exists a positive constant $c$ such that
    \begin{equation}\label{eq: projection-estimator}
        \inf_{\Pi}\sup_{\btheta\in \Gamma} \EE\left[\left\|\Pi(X) - \btheta\right\|_2^2\right]\le c\cdot \inf_{\hbtheta}\sup_{\btheta\in \Gamma} \EE\left[\left\|\hbtheta(X) - \btheta\right\|_2^2\right],
    \end{equation}
    where $\Pi$ denotes the class of projection estimators, and $\btheta$ denotes any estimators. Then we have 
    $$n_\gof(\Gamma, 2\sqrt{c}\epsilon)^2\le \frac{81n_\est(\Gamma, \epsilon)}{\epsilon^2}.$$
\end{proposition}
We remind that~\cite{donoho1990minimax} famously showed that~\eqref{eq: projection-estimator}
holds for QCO sets. Thus, the result above strictly generalizes second inequality in \pref{prop:
gof-est-upper}.
\begin{proof}
    We let $D(\Gamma, \sqrt{c}\epsilon)$ to be the Kolmogorov dimension of set $\Gamma$ at scale $\sqrt{c}\epsilon$ (the definition of Kolmogorov dimension is given in \pref{def: kolmogorov}). We use $\calS^\star$ to denote the $D(\Gamma, 3\epsilon)$-dimensional subspace which achieves the minimizer in the definition of Kolmogorov dimension \pref{def: kolmogorov}, and $\Pi_{\calS^\star}$ is the projection into $\calS^\star$. According to \cite{baraud2002non}, as long as $n\ge 9\sqrt{D(\Gamma, \sqrt{c}\epsilon)} / (c\epsilon^2)$, there exists a constant $c$ such that the testing scheme
    $$\psi(X) = \mathbb{I}\left[\left\|\Pi_{\calS^\star}\left[\frac{1}{n}\sum_{i=1}^n X_i\right]\right\|_2^2\ge c\right],$$
    satisfies 
    $$\PP_{X\simiid \calN(\zero, I)}(\psi(X) = 1)\le \frac{1}{4}\qquad\text{and}\quad \PP_{X\simiid \calN(\btheta, I)}(\psi(X) = 0)\le \frac{1}{4}\quad \text{for any }\|\btheta\|_2\ge 2\sqrt{c}.$$
    This implies that 
    $$n_\gof(\Gamma, 2\sqrt{c}\epsilon)\le \frac{9\sqrt{D(\Gamma, \sqrt{c}\epsilon)}}{c\epsilon^2}.$$

    Next, according to \pref{eq: projection-estimator},
    \begin{align*} 
        \inf_{\hbtheta}\sup_{\btheta\in \Gamma} \EE\left[\left\|\hbtheta(X) - \btheta\right\|_2^2\right] & \ge \frac{1}{c}\cdot \inf_{\Pi}\sup_{\btheta\in \Gamma} \EE\left[\left\|\Pi(X) - \btheta\right\|_2^2\right]\\
        & \ge \frac{1}{c}\cdot \inf_{d}\left\{\inf_{\Pi_d}\sup_{\btheta\in \Gamma}\left\{\|\Pi_d(\btheta) - \btheta\|_2^2 + \frac{d}{n}\right\}\right\}.
    \end{align*}
    where $\inf_{\Pi}$ denotes infimum over all projections. Hence when $n < D(\Gamma, \sqrt{c}\epsilon)/(c\epsilon^2)$, for any integer $d < D(\Gamma, \sqrt{c}\epsilon)$, 
    $$\inf_{\Pi_d}\sup_{\btheta\in \Gamma}\left\{\|\Pi_d(\btheta) - \btheta\|_2^2 + \frac{d}{n}\right\} > \inf_{\Pi_d}\sup_{\btheta\in \Gamma}\|\Pi_d(\btheta) - \btheta\|_2^2\ge (\sqrt{c}\epsilon)^2 = c\epsilon^2,$$
    and when $d \ge D(\Gamma, \sqrt{c}\epsilon)$, 
    $$\inf_{\Pi_d}\sup_{\btheta\in \Gamma}\left\{\|\Pi_d(\btheta) - \btheta\|_2^2 + \frac{d}{n}\right\}\ge \frac{d}{n} > c\epsilon^2.$$
    Therefore, we obtain that 
    $$\inf_{\hbtheta}\sup_{\btheta\in \Gamma} \EE\left[\left\|\hbtheta(X) - \btheta\right\|_2^2\right] > \frac{c\epsilon^2}{c} = \epsilon^2,$$
    which implies that 
    $$81n_\est(\Gamma, \epsilon)\ge \frac{81D(\Gamma, \sqrt{c}\epsilon)}{c\epsilon^2}\ge \epsilon^2\cdot n_\gof(\Gamma, 2\sqrt{c}\epsilon)^2.$$
\end{proof}

    We note that we use~\cite{donoho1990minimax} to show that a quadratically convex and orthosymmetric set $\Theta$ has a lower bound on the estimation sample complexity $n_{\est} \gtrsim \frac{D(\epsilon)}{ \epsilon^2}$. This result would automatically follow if we could show a geometric result that any such set necessarily contains $\gtrsim D(\epsilon)$ orthogonal vectors of length $\epsilon$ (or, equivalently, contains a ball of dimension $D$ and radius $\epsilon$), since then the lower bound would follow from standard results on Gaussian location model, e.g.~\cite[Theorem 30.1]{polyanskiy2025information}. We summarize this observation into the following conjecture. It is an interesting question in convex geometry to prove or disprove it.

\paragraph*{Conjecture} There exist universal positive constants $c_1$ and $c_2$ such that for any orthosymmetric, compact, convex and quadratically convex set $\Theta$, there exist $D(\Theta, c_1\epsilon)$ orthogonal vectors of length $(c_2\epsilon)$ all contained within $\Theta$.


\section{Analysis of LFHT for General Convex Sets}
\subsection{Proof of \pref{thm: lfht-orthosymmetric}}\label{sec: lfht-orthosymmetric-app-l}
\begin{proof}[Proof of \pref{thm: lfht-orthosymmetric}]
    Same as the proof of \pref{thm: gof-est}, there exists $\btheta^* = (\theta_1^*, \cdots, \theta_D^*)\in\Gamma$ such that
    \begin{enumerate}[label=(\alph*)]
        \item For any $1\le i\le D$, $|\theta_i^*|\le 2\sqrt{\frac{\log(2D)}{n_\est(\Gamma, \sqrt{2}\epsilon) - 1}}$. \label{theta-star-a-lfht}
        \item $\sum_{i=1}^d(\theta_i^*)^2 \ge \epsilon^2$. \label{theta-star-b-lfht}
        \item $\sum_{i=1}^d(\theta_i^*)^2 \le 16\epsilon^2$. \label{theta-star-c-lfht}
    \end{enumerate}
    Next, we use these three properties to bound the LFHT testing region of set $\Gamma$. First of all, according to \cite[Proposition 1]{gerber2024likelihood}, if $(m, n)$ lies in the LFHT testing region, we must have 
    \begin{equation}\label{eq: lfht-first-two}
        m\ge\frac{1}{\epsilon^2}\quad\text{and}\quad n\ge n_\gof(\Gamma, \epsilon)\ge \frac{1}{8\sqrt{\log(2D)}}\cdot \frac{\sqrt{n_\est(\Gamma, \sqrt{2}\epsilon)}}{\epsilon},
    \end{equation}
    where the last inequality follows from \pref{thm: gof-est}. Hence we only need to verify that if $(m, n)$ lies in the LFHT testing region, then
    \begin{equation}\label{eq: lfht-orthosymmetric-mn-ge}
        mn\ge \frac{n_\est(\Gamma, \sqrt{2}\epsilon) - 1}{3072\epsilon^2\cdot \log(2D)}.
    \end{equation}

    Without loss of generality, we assume 
    \begin{equation}\label{eq: lfht-orthosymmetric-m}
        m\le \frac{\sqrt{n_\est(\Gamma, \sqrt{2}\epsilon) - 1}}{192\sqrt{\log(2D)}},
    \end{equation}
    otherwise the inequality \pref{eq: lfht-orthosymmetric-mn-ge} follows from the lower bound to $n$ in \pref{eq: lfht-first-two}. Next we will verify that if $(m, n)$ satisfies 
    \begin{equation}\label{eq: lfht-orthosymmetric-mn}
        mn \le \frac{n_\est(\Gamma, \sqrt{2}\epsilon) - 1}{3072\epsilon^2\cdot \log(2D)},
    \end{equation}
    then \pref{eq: lfht-objective} fails.
    
    For any fixed distribution $\mu\in \Delta(\Gamma)$, we define distribution $\PP_{0, XYZ}$ to be the distribution of $(\bX, \bY, \bZ)$ sampled according to the following way: first sample $\btheta\sim \mu$, then sample $\bX = (\bX^{1:n})\simiid \calN(\btheta, I_D)$, $\bY = (\bY^{1:n})\simiid \calN(\zero, I_D)$ and $\bZ = (\bZ^{1:m})\simiid \calN(\btheta, I_D)$. And we define distribution $\PP_{1, XYZ}$ to be the distribution of $(\bX, \bY, \bZ)$ sampled according to the following way: first sample $\btheta\sim \mu$, then sample $\bX = (\bX^{1:n})\simiid \calN(\btheta, I_D)$, $\bY = (\bY^{1:n})\simiid \calN(\zero, I_D)$ and $\bZ = (\bZ^{1:m})\simiid \calN(\zero, I_D)$. Similarly, we can define distribution $\PP_{0, XZ}, \PP_{1, XZ}, \PP_{0, X}, \PP_{1, X}$, and also conditional distribution $\PP_{0, Z|X}$ and $\PP_{1, Z|X}$. Then \cite[Lemma 5]{gerber2024likelihood} gives that for any such $\mu$ supported in $\Gamma\backslash B_2(\epsilon)$, 
    \begin{equation}\label{eq: condition-lfht-testing-inequality-orthosymmetric}
        \inf_{\psi} \max_{i\in \{0, 1\}} \sup_{P\in H_i} \PP(\psi(X, Y, Z)\neq i)\ge \frac{1}{2}\left(1 - \tv(\PP_{0, XYZ}, \PP_{1, XYZ})\right).
    \end{equation}
    In the following proof, we choose distribution $\mu\in \Delta(\Gamma)$ to be the product of symmetric ternary distributions, i.e. for $\theta_i = (\theta_1, \cdots, \theta_D)\in \Gamma$,
    $$\mu(\theta) = \prod_{j=1}^D \mu_j(\theta_j)\quad \text{where}\quad \mu_j = \frac{\delta_{\theta_j^\star}}{2} + \frac{\delta_{-\theta_j^\star}}{2}.$$
    According to \ref{theta-star-c-lfht} and the property of orthosymmetric, we have $\mu\in \Delta(\Gamma)$. We can also verify that for any $\btheta = (\theta_1, \cdots, \theta_D)$ which belongs to the support of $\mu$, we have $|\theta_j| = |\theta_j^\star|$ for any $j\in [D]$. Hence 
    $$\|\btheta\|_2^2 = \sum_{j=1}^D (\theta_j)^2 = \sum_{j=1}^D(\theta_j^\star)^2 \ge\epsilon^2,$$
    where the last inequality uses \ref{theta-star-b-lfht}. This verifies that $\mu$ is supported in $\Gamma\backslash B_2(\epsilon)$.
    
    Next, we will calculate the TV distance $\tv(\PP_{0, XYZ}, \PP_{1, XYZ})$:
    \begin{align*}
        \tv(\PP_{0, XYZ}, \PP_{1, XYZ})^2 & = \tv(\PP_{0, XZ}, \PP_{1, XZ})^2 \le \KL(\PP_{0, XZ}\ \|\ \PP_{1, XZ})\\
        & = \KL(\PP_{0, Z\mid X}\ \|\ \PP_{1, Z\mid X}\mid \PP_{0, X}) + \KL(\PP_{0, X}\ \|\ \PP_{1, X})\\
        & = \KL(\PP_{0, Z\mid X}\ \|\ \PP_{1, Z\mid X}\mid \PP_{0, X})\le \chi^2(\PP_{0, Z\mid X}\ \|\ \PP_{1, Z\mid X}\mid \PP_{0, X}). \numberthis\label{eq: tv-kl-chi-2-orthosymmetric}
    \end{align*}
    Hence in order to bound the above TV distance, we only need to upper bound the conditional $\chi^2$-divergence in the right hand side. Using $\varphi_{\btheta}(\cdot)$ to denote the density function of $\calN(\btheta, I_D)$, according to Ingster's trick \cite{ingster1987minimax}, we have
    \begin{align*} 
        &\hspace{-0.5cm} \chi^2(\PP_{0, Z\mid X}\ \|\ \PP_{1, Z\mid X}\mid \PP_{0, X}) + 1\\
        & = \EE_{\bX\sim \PP_{0, X}}\left[\EE_{\btheta\mid \bX, \btheta'\mid \bX}\left[\int_{(\RR^D)^m} \prod_{t=1}^m \frac{\varphi_\btheta(\bz^t)\varphi_{\btheta'}(\bz^t)}{\varphi_{\zero}(\bz^t)}d\bz_1\cdots d\bz_m\mid \bX\right]\right]\\
        & = \EE_{\bX\sim \PP_{0, X}}\left[\EE_{\btheta\mid \bX, \btheta'\mid \bX}[\exp\left(m\langle \btheta, \btheta'\rangle\right)\mid \bX]\right]\\
        & = \EE_{\bX\sim \PP_{0, X}}\left[\prod_{j=1}^D \left(\PP(\theta_j = \theta_j'\mid \bX)\cdot \exp(m(\theta_i^\star)^2) + \PP(\theta_j \neq \theta_j'\mid \bX)\cdot \exp(-m(\theta_i^\star)^2)\right)\right]\\
        & = \prod_{j=1}^D\EE_{\bX\sim \PP_{0, X}}\left[\PP(\theta_j = \theta_j'\mid \bX)\cdot \exp(m(\theta_i^\star)^2) + \PP(\theta_j \neq \theta_j'\mid \bX)\cdot \exp(-m(\theta_i^\star)^2)\right], \numberthis \label{eq: chi-square-lfht-orthosymmetric}
    \end{align*}
    where $\btheta, \btheta'$ are i.i.d. sampled according to $\PP(\btheta\mid \bX)$, and the last equation uses the fact that conditioned on $\bX$, we have $(\theta_j, \theta_j')$ independent to each other for any $j\in [D]$. In the following, we write $\bX = \bX^{1:n}$ which denotes the $n$ samples, and we further denote $\bX^i = (X^i_1, \cdots, X^i_D)$, where $X^i_j$ denotes the $j$-th coordinate of $\bX^i$. We notice that $\theta_j$ only depends on $X^{1:n}_j = (X^1_j, \cdots, X^n_j)$. According to Bayes rule, we can calculate
    \begin{align*} 
        \PP(\theta_j = 1\mid \bX) & = \frac{\Pr(X_j^{1:n}, \eta = 1)}{\Pr(X_j^{1:n}, \eta_j = -1) + \Pr(X_j^{1:n}, \eta_j = 1)}\\
        & = \frac{\prod_{i=1}^n\exp\left(-(X_j^i - \theta_j^\star)^2/2\right)}{\prod_{i=1}^n\exp\left(- (X_j^i - \theta_j^\star)^2/2\right) + \prod_{i=1}^n\exp\left(-(X_j^i + \theta_j^\star)^2/2\right)}\\
        & = \frac{\exp\left(\theta_j^\star\cdot \sum_{i=1}^n X_j^i\right)}{\exp\left(\theta_j^\star\cdot \sum_{i=1}^n X_j^i\right) + \exp\left(-\theta_j^\star\cdot \sum_{i=1}^n X_j^i\right)}. \numberthis \label{eq: positive-eta-orthosymmetric}
    \end{align*}
    Similarly, we get 
    \begin{equation}\label{eq: negative-eta-orthosymmetric}
        \PP(\theta_j = -1 \mid \bX) = \frac{\exp\left(-\theta_j^\star\cdot \sum_{i=1}^n X_j^i\right)}{\exp\left(\theta_j^\star\cdot \sum_{i=1}^n X_j^i\right) + \exp\left(-\theta_j^\star\cdot \sum_{i=1}^n X_j^i\right)}.
    \end{equation}
    In the following, we use $[0, 1]$-valued random variable $p_j(\bX)$ to denote 
    $$p_j(\bX) = \PP(\theta_j = 1\mid \bX),\qquad \forall j\in [D].$$
    We further notice that $\theta_j'$ and $\theta_j$ are i.i.d. conditioned on $\bX$. Hence we obtain 
    \begin{align*} 
        &\hspace{-0.5cm} \PP(\theta_j = \theta_j'\mid \bX)\cdot \exp(m(\theta_i^\star)^2) + \PP(\theta_j \neq \theta_j'\mid \bX)\cdot \exp(-m(\theta_i^\star)^2)\\
        & = (p_j(\bX)^2 + (1 - p_j(\bX))^2)\exp(m(\theta_j^\star)^2) + 2p_j(\bX)(1 - p_j(\bX))\exp(-m(\theta_j^\star)^2)\\
        & = \exp(m(\theta_j^\star)^2) + \exp(-m(\theta_j^\star)^2) - 1 + \frac{(1 - 2p_j(\bX))^2}{2}\cdot \left(\exp(m(\theta_j^\star)^2) - \exp(-m(\theta_j^\star)^2)\right). \numberthis \label{eq: p-eta-eta'-orthosymmetric}
    \end{align*}
    Next according to \pref{eq: lfht-orthosymmetric-m} and \ref{theta-star-a-lfht}, we have for any $j\in [D]$, $m(\theta_j^\star)^2 \le 1$, which implies that  
    \begin{align*} 
        &\quad \exp(m(\theta_j^\star)^2) + \exp(-m(\theta_j^\star)^2) - 2\le 1 + 4 m^2(\theta_j^\star)^4\\
        &\text{and}\qquad\quad \exp(m(\theta_j^\star)^2) - \exp(-m(\theta_j^\star)^2)\le 4m(\theta_j^\star)^2. \numberthis \label{eq: lfht-orthosymmetric-m-bound}
    \end{align*}
    We further notice that 
    \begin{align*} 
        \frac{(1 - 2p_j(\bX))^2}{2} & = \frac{(\exp\left(\theta_j^\star\cdot \sum_{i=1}^n X_j^i\right) - \exp\left(-\theta_j^\star\cdot \sum_{i=1}^n X_j^i\right))^2}{2(\exp\left(\theta_j^\star\cdot \sum_{i=1}^n X_j^i\right) + \exp\left(-\theta_j^\star\cdot \sum_{i=1}^n X_j^i\right))^2}\\
        & \le \frac{(\exp\left(\theta_j^\star\cdot \sum_{i=1}^n X_j^i\right) - \exp\left(-\theta_j^\star\cdot \sum_{i=1}^n X_j^i\right))^2}{8},
    \end{align*}
    which implies that 
    \begin{align*} 
        \EE_{\bX\sim \PP_{0, X}}\left[\frac{(1 - 2p_j(\bX))^2}{2}\right] & \le \EE_{\bX\sim \PP_{0, X}}\left[\frac{(\exp\left(\theta_j^\star\cdot \sum_{i=1}^n X_j^i\right) - \exp\left(-\theta_j^\star\cdot \sum_{i=1}^n X_j^i\right))^2}{8}\right]\\
        & = \frac{\exp(4n(\theta_j^\star)^2) - 1}{8},
    \end{align*}
    where the last equation uses the equation $\EE[\exp(\alpha X)] = \exp(\alpha\mu + \alpha^2\sigma^2/2)$ for $X\sim \calN(\mu, \sigma)$, and also the way of sampling $\bX$ from $\PP_{0, X}$. For $j\in [D]$, if $n(\theta_j^\star)^2\le 1$, then according to the inequality $\exp(x) - 1\le 2x$ for $0\le x\le 1$, we have $\EE_{\bX\sim \PP_{0, X}}\left[(1 - 2p_j(\bX))^2/2\right]\le n(\theta_j^\star)^2.$
    If $n(\theta_j^\star)^2\ge 1$, since $p_j(\bX)\in [0, 1]$, we have $\EE_{\bX\sim \PP_{0, X}}\left[(1 - 2p_j(\bX))^2/2\right]\le 1\le n(\theta_j^\star)^2$. Hence we obtain that for any $j\in [D]$,  
    $$\EE_{\bX\sim \PP_{0, X}}\left[\frac{(1 - 2p_j(\bX))^2}{2}\right]\le n(\theta_j^\star)^2.$$
    Bring this inequality and \pref{eq: lfht-orthosymmetric-m-bound} back to \pref{eq: chi-square-lfht-orthosymmetric}, we obtain that 
    \begin{align*} 
        \chi^2(\PP_{0, Z\mid X}\ \|\ \PP_{1, Z\mid X}\mid \PP_{0, X}) + 1 & \le \prod_{j=1}^D \left(1 + 4m^2(\theta_j^\star)^4 + 4mn(\theta_j^\star)^4\right)\\
        & \le \left(1 + \frac{4m^2 + 4mn}{D}\cdot \sum_{j=1}^D (\theta_j^\star)^4\right)^D,
    \end{align*}
    where the last inequality uses the Jensen's inequality. Therefore, if \pref{eq: lfht-orthosymmetric-mn} holds, then together with \pref{eq: lfht-orthosymmetric-m} and also \ref{theta-star-a-lfht} and \ref{theta-star-c-lfht}, we have 
    $$(4m^2 + 4mn)\cdot \sum_{j=1}^D (\theta_j^\star)^4\le (4m^2 + 4mn)\cdot \sum_{j=1}^D (\theta_j^\star)^2\cdot \max_{j\in [D]}|\theta_j^\star|^2\le \frac{1}{6},$$
    which implies that 
    $$\tv(\PP_{0, XYZ}, \PP_{1, XYZ})\le \sqrt{\chi^2(\PP_{0, Z\mid X}\ \|\ \PP_{1, Z\mid X}\mid \PP_{0, X})} \le \sqrt{(1 + 1/(6D))^D - 1}\le \sqrt{e^{1/6} - 1}\le \frac{1}{2}$$
    Hence according to \pref{eq: condition-lfht-testing-inequality-orthosymmetric}, \pref{eq: lfht-objective} fails.
\end{proof}

\subsection{Proof of \pref{thm: quad-lfht}}\label{sec: quad-uncondintional-app}
\begin{proof}[Proof of \pref{thm: quad-lfht}]
    To verify \pref{eq: lfht-objective}, we only need to show that 
    $$\sup_{P\in H_i} \PP(\psi(X, Y, Z) \neq i)\le \frac{1}{4},\quad \forall i\in \{0, 1\}.$$
    Without loss of generality, we only prove the above inequality for $i = 0$, i.e. when $\pz = \px$, we always have 
    \begin{equation}\label{eq: orthosymmetric-testing-scheme}
        \PP(T_\lf\ge 0)\le \frac{1}{4}.
    \end{equation}
    The proof of cases where $i = 1$ follows similarly.

    Without loss of generality, we assume the projection $\Pi_d$ is onto the first $d$-coordinates. Assume 
    $$\hthetax = (\hthetax)_{1:D},\quad  \hthetay = (\hthetay)_{1:D},\quad  \hthetaz = (\hthetaz)_{1:D},\quad  \px = (\px)_{1:D},\quad  \py = (\py)_{1:D}\quad \text{and}\quad \pz = (\pz)_{1:D}$$
    For every $t\in [d]$, we let 
    $$u_t = \left((\hthetax)_t - (\hthetaz)_t\right)^2 - \left((\hthetay)_t - (\hthetaz)_t\right)^2.$$ 
    Then we have 
    $$(\hthetax)_t\sim \calN\left((\px)_t, \frac{1}{n}\right), \qquad (\hthetay)_t\sim \calN\left((\py)_t, \frac{1}{n}\right)\quad\text{and}\quad (\hthetaz)_t\sim \calN\left((\pz)_t, \frac{1}{n}\right).$$
    Hence we get 
    $$\EE[u_t] = [(\px)_t]^2 - [(\py)_t]^2 - 2((\px)_t - (\py)_t)(\pz)_t$$
    and 
    \begin{align*} 
        \var(u_t) & = \EE[(u_t)^2] - (\EE[u_t])^2\\
        & = \frac{4}{n}((\px)_t - (\pz)_t)^2 + \frac{4}{n}((\py)_t - (\pz)_t)^2 + \frac{4}{m}((\px)_t - (\py)_t)^2 + \frac{4}{n^2} + \frac{8}{mn}.
    \end{align*}
    Notice that $\pz = \px$, we get 
    \begin{align*} 
        \EE[u_t] = - ((\px)_t - (\py)_t)^2\quad\text{and}\quad 
        \var(u_t) = \left(\frac{4}{m} + \frac{4}{n}\right)((\px)_t - (\py)_t)^2 + \frac{4}{n^2} + \frac{8}{mn}.
    \end{align*}
    Next notice that $T_\lf = \sum_{t\in [d]} u_t$, we obtain 
    \begin{align*}
        \EE[T_\lf] & = -\sum_{t\in [d]}((\px)_t - (\py)_t)^2\quad\text{and}\quad \var(T_\lf) = \left(\frac{4}{m} + \frac{4}{n}\right)\cdot \sum_{t\in [d]}((\px)_t - (\py)_t)^2 + \frac{4d}{n^2} + \frac{8d}{mn}.
    \end{align*}
    According to our choice of $d$ in \pref{eq: quadratically-convex-kolmogorov}, we have 
    \begin{align*}
        \sum_{t\in [d]}((\px)_t - (\py)_t)^2 & = \|\px - \py\|_2^2 - \sum_{t\not\in [d]} ((\px)_t - (\py)_t)^2\\
        & \ge \epsilon^2 - 2\sum_{t\not\in [d]}((\px)_t)^2 - 2\sum_{t\not\in [d]}((\px)_t)^2\ge \epsilon^2 - \frac{2\epsilon^2}{9} - \frac{2\epsilon^2}{9} \ge \frac{\epsilon^2}{2}.
    \end{align*}
    Therefore, if $m\ge 96/\epsilon^2$, $n\ge 96\sqrt{d} / \epsilon^2$ and $mn\ge 768d/\epsilon^4$, we have 
    $$\var(T_\lf)\le \frac{1}{4}\cdot (\EE[T_\lf])^2.$$
    Therefore, according to Chebyshev's inequality, we obtain that 
    $$\Pr(T_\lf\ge 0)\le \frac{\var(T_\lf)}{(\EE[T_\lf])^2}\le \frac{1}{4},$$
    which verifies \pref{eq: orthosymmetric-testing-scheme}.

\end{proof}

\section{Analysis of LFHT for $\ell_p$ Bodies}\label{sec: app-lfht}

\subsection{Proof of \pref{thm: lp-lfht-ub}}\label{sec: app-lfht-u}
The proof of theorem \pref{thm: lp-lfht-ub} requires the following lemma:
\begin{lemma}\label{lem: coordinates}
    Suppose that $1\le p\le 2$. For $\px, \py\in \Gamma$ with $\|\px - \py\|_2\ge \epsilon$, with probability at least $1 - \delta$, the set $T$ calculated in \ref{item: coordinates} satisfies
    \begin{enumerate}[label=(\alph*)]
        \item \label{item: sum} $\sum_{t\in T}\left((\px)_t - (\py)_t\right)^2\ge \frac{\epsilon^2}{2}$.
        \item \label{item: card} $\card(T)\le 2d_u(\Gamma, n, \epsilon)$.
    \end{enumerate}
\end{lemma}
\begin{proof}[Proof of \pref{lem: coordinates}]
    Since $(\hthetax)_t$, $(\hthetay)_t$ are empirical estimation of $(\px)_t$ and $(\py)_t$, with $n$ samples each:
    $$(\bX_1)_t, \cdots, (\bX_n)_t\simiid \calN((\px)_t, 1),\quad \text{and}\quad (\bY_1)_t, \cdots, (\bY_n)_t\simiid \calN((\py)_t, 1),\qquad \forall t\in [D],$$ 
    we have $(\hthetax)_t\sim \calN((\px)_t, 1/n)$ and $(\hthetay)_t\sim \calN((\py)_t, 1/n)$. Hence according to \cite[Proposition 2.1.2]{vershynin2018high}, we have with probability at least $1 - \delta$, for any $t\in [D]$, both of the following events hold:
    \begin{equation}\label{eq: condition-hp}
        \left|(\hthetax)_t - (\px)_t\right|\le \sqrt{\frac{2\log(4D/\delta)}{n}}\quad\text{and}\quad \left|(\hthetay)_t - (\py)_t\right|\le \sqrt{\frac{2\log(4D/\delta)}{n}}.
    \end{equation}
    In the rest of the proof, we assume both events in \pref{eq: condition-hp} holds for any $t\in [D]$, and we will prove the two conditions in \pref{lem: coordinates}.

    We first verify \ref{item: sum}. According to our construction of set $T_2$ in \pref{eq: def-T-1-T-2} for any $t\in T_2$, we have $t > d_u(\Gamma, n, \epsilon)$.  Hence for any $t\in T_2$, the coefficient $a_t$ in \pref{eq: ell-p-body} satisfies $a_t \le a_{d_u(\Gamma, n, \epsilon) + 1}$. Since $\px, \py\in \Gamma$, we obtain
    $$\sum_{t\in T_2} \frac{|(\px)_t|^p}{(a_{d_u(\Gamma, n, \epsilon) + 1})^p}\le \sum_{t\in T_2} \frac{|(\px)_t|^p}{(a_t)^p}\le \sum_{t\in [D]} \frac{|(\px)_t|^p}{(a_t)^p}\le 1,$$
    and
    $$\sum_{t\in T_2} \frac{|(\py)_t|^p}{(a_{d_u(\Gamma, n, \epsilon) + 1})^p}\le \sum_{t\in T_2} \frac{|(\py)_t|^p}{(a_t)^p}\le \sum_{t\in [D]} \frac{|(\py)_t|^p}{(a_t)^p}\le 1,$$
    which implies that 
    \begin{equation}\label{eq: upper-bound-p-p}
        \sum_{t\in T_2} \left|(\px)_t - (\py)_t\right|^p\le 2\sum_{t\in T_2}\left[|(\px)_t|^p + |(\py)_t|^p\right]\le 4(a_{d_u(\Gamma, n, \epsilon) + 1})^p,
    \end{equation}
    where the first inequality uses the fact that $|a - b|^p\le 2(|a|^p + |b|^p)$ for any $a, b\in \RR$ and  $1\le p\le 2$. Therefore, 
    \begin{align*}
        &\hspace{-0.5cm} \sum_{t\in T_2} \left((\px)_t - (\py)_t\right)^2\cdot \mathbb{I}\left[|(\px)_t - (\py)_t| < 6\sqrt{\frac{2\log(4D/\delta)}{n}}\right]\\
        & \le \left[\sum_{t\in T_2} \left|(\px)_t - (\py)_t\right|^p\right]\cdot \left(6\sqrt{\frac{2\log(4D/\delta)}{n}}\right)^{2-p}\\
        & \le 4(a_{d_u(\Gamma, n, \epsilon) + 1})^p n^{\frac{p-2}{2}}\cdot \sqrt{72\log(4D/\delta)}^{2-p}\\
        & \le 4(a_{d_u(\Gamma, n, \epsilon) + 1})^p n^{\frac{p-2}{2}}\cdot 72\log(4D/\delta)
    \end{align*}
    According to the definition of $d_u(\Gamma, n, \epsilon)$ in \pref{eq: def-d-u}, we have 
    \begin{equation}\label{eq: condition-n-epsilon}
        (a_{d_u(\Gamma, n, \epsilon) + 1})^p n^{\frac{p-2}{2}} < \frac{\epsilon^2}{576\log(4D/\delta)},
    \end{equation}
    which implies that
    $$\sum_{t\in T_2}\left((\px)_t - (\py)_t\right)^2\cdot \mathbb{I}\left[|(\px)_t - (\py)_t| < 6\sqrt{\frac{2\log(4D/\delta)}{n}}\right]\le \frac{\epsilon^2}{2}.$$
    According to the construction of set $T$, we have for any $t\in [D]\backslash T$, 
    $$|(\hthetax)_t - (\hthetay)_t|\le 4\sqrt{\frac{2\log(4D/\delta)}{n}}.$$
    Hence the conditions in \pref{eq: condition-hp} indicates that 
    $$|(\px)_t - (\py)_t|\le |(\hthetax)_t - (\hthetay)_t| + 2\sqrt{\frac{2\log(4D/\delta)}{n}}\le 6\sqrt{\frac{2\log(4D/\delta)}{n}},$$
    which implies that
    \begin{align*} 
      &\hspace{-0.5cm} \sum_{t\in T} \left((\px)_t - (\py)_t\right)^2\\
      & \ge \sum_{t=1}^D \left((\px)_t - (\py)_t\right)^2 - \sum_{t\in T_2}\left((\px)_t - (\py)_t\right)^2\cdot \mathbb{I}\left[|(\px)_t - (\py)_t| < 6\sqrt{\frac{2\log(4D/\delta)}{n}}\right]\\
      & \ge\frac{\epsilon^2}{2}.
    \end{align*}

    We next verify \ref{item: card}. Since $\card(T_1) = d_u(\Gamma, n, \epsilon)$ and $T = T_1\cup T_3$ according to its definition, we only need to verify $\card(T_3)\le d_u(\Gamma, n, \epsilon)$. We further notice that the conditions in \pref{eq: condition-hp} gives that for any $t\in T_3$, 
    $$|(\px)_t - (\py)_t|\ge |(\hthetax)_t - (\hthetay)_t| - 2\sqrt{\frac{2\log(4D/\delta)}{n}}\ge 2\sqrt{\frac{2\log(4D/\delta)}{n}}.$$
    Next, since $T_3\subseteq T_2$, \pref{eq: upper-bound-p-p} indicates that 
    $$\sum_{t\in T_3} \left|(\px)_t - (\py)_t\right|^p\le \sum_{t\in T_2} \left|(\px)_t - (\py)_t\right|^p\le 4(a_{d_u(\Gamma, n, \epsilon) + 1})^p,$$
    which implies that 
    \begin{align*}
        \card(T_3) & \le \frac{4(a_{d_u(\Gamma, n, \epsilon) + 1})^p}{\left(2\sqrt{\frac{2\log(4D/\delta)}{n}}\right)^p}\le 2(a_{d_u(\Gamma, n, \epsilon) + 1})^p n^{p/2}\stackrel{(i)}{\le} \frac{2n\epsilon^2}{576\log(4D/\delta)^2}\le n\epsilon^2,
    \end{align*}
    where inequality $(i)$ uses \pref{eq: condition-n-epsilon}. Notice from \pref{eq: condition-mn} and also the condition that 
    $$mn \le \frac{1024d_u(\Gamma, n, \epsilon)}{\epsilon^4},$$
    we have 
    $$n\epsilon^2\le d_u(\Gamma, n, \epsilon),$$
    which implies that 
    $$\card(T) = \card(T_1) + \card(T_3) \le d_u(\Gamma, n, \epsilon) + d_u(\Gamma, n, \epsilon)\le 2d_u(\Gamma, n, \epsilon).$$
    Hence \ref{item: card} is verified.
\end{proof}
Now equipped with \pref{lem: coordinates}, we are ready to prove \pref{thm: lp-lfht-ub}.
\begin{proof}[Proof of \pref{thm: lp-lfht-ub}]
    To verify \pref{eq: lfht-objective}, we only need to show that 
    $$\sup_{P\in H_i} \PP(\psi(X, Y, Z) \neq i)\le \frac{1}{4},\quad \forall i\in \{0, 1\}.$$
    Without loss of generality, we only prove the above inequality for $i = 0$, i.e. when $\pz = \px$, we always have 
    $$\PP(T_\lf\ge 0)\le \frac{1}{4}.$$
    The proof of cases where $i = 1$ follows similarly.

    For every $t\in [D]$, we let 
    $$u_t = \left((\hthetax^2)_t - (\hthetaz)_t\right)^2 - \left((\hthetay^2)_t - (\hthetaz)_t\right)^2.$$ 
    And for simplicity, we set $N = n - n_0 = n - \lfloor n/2\rfloor$. Then we have 
    $$(\hthetax^2)_t\sim \calN\left((\px)_t, \frac{1}{N}\right), \quad (\hthetay^2)_t\sim \calN\left((\py)_t, \frac{1}{N}\right)\quad \text{and}\quad (\hthetaz)_t\sim \calN\left((\pz)_t, \frac{1}{m}\right).$$
    Therefore, since $(\bX^2, \bY^2, \bZ)\indep (\bX^1, \bY^1)$, we can calculate
    $$\EE[u_t\mid \bX^1, \bY^1] = [(\px)_t]^2 - [(\py)_t]^2 - 2((\px)_t - (\py)_t)(\pz)_t,$$
    and further
    \begin{align*}
        \var(u_t\mid \bX^1, \bY^1) & = \EE[u_t^2\mid \bX^1, \bY^1] - (\EE[u_t\mid \bX^1, \bY^1])^2\\
        & = \frac{4}{N}((\px)_t - (\pz)_t)^2 + \frac{4}{N}((\py)_t - (\pz)_t)^2 + \frac{4}{m}((\px)_t - (\py)_t)^2 + \frac{4}{N^2} + \frac{8}{mN}.
    \end{align*}
    When $\pz = \px$, we have 
    \begin{align*}
        \EE[u_t\mid \bX^1, \bY^1] & = - ((\px)_t - (\py)_t)^2\\
        \var(u_t\mid \bX^1, \bY^1) & = \left(\frac{4}{m} + \frac{4}{N}\right)\cdot ((\px)_t - (\py)_t)^2 + \frac{4}{N^2} + \frac{8}{mN}.
    \end{align*}
    According to our construction of set $T\subseteq [D]$, set $T$ is deterministic when conditioned on samples $\bX^1, \bY^1$. Therefore, we obtain that 
    \begin{align*}
        \EE[T_\lf\mid \bX^1, \bY^1] & = - \sum_{t\in T} ((\px)_t - (\py)_t)^2\\
        \var(T_\lf\mid \bX^1, \bY^1) & = \left(\frac{4}{m} + \frac{4}{N}\right)\cdot \left(\sum_{t\in T}((\px)_t - (\py)_t)^2\right) + \left(\frac{4}{N^2} + \frac{8}{mN}\right)\cdot |T|.
    \end{align*}
    When the conditions \ref{item: sum} and \ref{item: card} in \pref{lem: coordinates} holds, we have 
    $$\sum_{t\in T}((\px)_t - (\py)_t)^2\ge \frac{\epsilon^2}{2}\quad \text{and}\quad |T|\le 2d_u(\Gamma, n, \epsilon).$$
    Next we notice that $N = n - \lfloor n/2\rfloor \ge n/2$. Hence, when $(m, n)$ satisfies 
    $$m\ge \frac{32}{\epsilon^2},\quad n\ge \frac{32\sqrt{d_u(\Gamma, n, \epsilon)}}{\epsilon^2}\quad \text{and}\quad mn\ge \frac{1024d_u(\Gamma, n, \epsilon)}{\epsilon^2},$$
    which implies that 
    $$\var(T_\lf\mid \bX^1, \bY^1)\le \frac{7(\EE[T_\lf\mid \bX^1, \bY^1])^2}{32}.$$
    Therefore, according to Chebyshev's inequality, we obtain that 
    $$\Pr(T_\lf\ge 0\mid \bX^1, \bY^1)\le \frac{\var(T_\lf\mid \bX^1, \bY^1)}{(\EE[T_\lf\mid \bX^1, \bY^1])^2}\le \frac{7}{32}.$$

    Finally, we notice that according to \pref{lem: coordinates} with $\delta = 1/32$, with probability at least $31/32$, conditions \ref{item: sum} and \ref{item: card} both holds, which implies that 
    $$\Pr(T_\lf\ge 0)\le \Pr(T_\lf\ge 0\mid \bX^1, \bY^1) + \frac{1}{32} \le \frac{7}{32} + \frac{1}{32} =  \frac{1}{4}.$$
\end{proof}

\subsection{Proof of \pref{thm: lower-bound-lfht}}\label{sec: app-lfht-l}
We present the proof of \pref{thm: lower-bound-lfht} in this section. 
\begin{proof}[Proof of \pref{thm: lower-bound-lfht}]
    According to \cite[Proposition 3]{baraud2002non}, the condition of goodness-of-fit test (the condition where \pref{eq: condition-gof} holds) is 
    $$n\ge n_\gof\triangleq \frac{\sqrt{d_l(\Gamma, n, \epsilon)}}{2\epsilon^2}.$$
    According to \cite[Proposition 1]{gerber2024likelihood}, if there exists a test which satisfies \pref{eq: lfht-objective}, then $(m, n)$ has to satisfies 
    $$m\ge \frac{1}{\epsilon^2},\quad \text{and}\quad n\ge n_\gof = \frac{\sqrt{d_l(\Gamma, n, \epsilon)}}{2\epsilon^2}.$$
    Therefore, we only need to verify the third condition, i.e. 
    $$mn\gtrsim \frac{d_l(\Gamma, n, \epsilon)}{\epsilon^4}.$$
    If $m\ge n$, since $n$ satisfies $n\ge \sqrt{d(\Gamma, n, \epsilon)}/(2\epsilon^2)$, we have 
    $$mn\ge n^2\ge \frac{d_l(\Gamma, n, \epsilon)}{4\epsilon^2}$$
    and the third condition of \pref{eq: lb-lfht} is automatically satisfied. In the following, we assume $m < n$, and we will verify the third condition of \pref{eq: lb-lfht}. Above all, in the following we assume 
    \begin{equation}\label{eq: condition-m-n}
        n > m\ge \frac{1}{\epsilon^2},\quad \text{and}\quad n\ge \frac{\sqrt{d_l(\Gamma, n, \epsilon)}}{2\epsilon^2},
    \end{equation}
    and we will show that if 
    \begin{equation} \label{eq: condition-mn-lower}
        mn < \frac{d_l(\Gamma, n, \epsilon)}{96\epsilon^4},
    \end{equation}
    then \pref{eq: lfht-objective} fails.

    In the following, we assume \pref{eq: condition-mn-lower} holds, and we let $d = d_l(\Gamma, n, \epsilon)$. For any fixed distribution $\mu\in \Delta(\RR^D)$, we define distribution $\PP_{0, XYZ}$ to be the distribution of $(\bX, \bY, \bZ)$ sampled according to the following way: first sample $\btheta\sim \mu$, then sample $\bX = (\bX^{1:n})\simiid \calN(\btheta, I_D)$, $\bY = (\bY^{1:n})\simiid \calN(\zero, I_D)$ and $\bZ = (\bZ^{1:m})\simiid \calN(\btheta, I_D)$. And we define distribution $\PP_{1, XYZ}$ to be the distribution of $(\bX, \bY, \bZ)$ sampled according to the following way: first sample $\btheta\sim \mu$, then sample $\bX = (\bX^{1:n})\simiid \calN(\btheta, I_D)$, $\bY = (\bY^{1:n})\simiid \calN(\zero, I_D)$ and $\bZ = (\bZ^{1:m})\simiid \calN(\zero, I_D)$. Similarly, we can define distribution $\PP_{0, XZ}, \PP_{1, XZ}, \PP_{0, X}, \PP_{1, X}$, and also conditional distribution $\PP_{0, Z|X}$ and $\PP_{1, Z|X}$. Then \cite[Lemma 5]{gerber2024likelihood} gives that for any such $\mu$, 
    \begin{equation}\label{eq: condition-lfht-testing-inequality}
        \inf_{\psi} \max_{i\in \{0, 1\}} \sup_{P\in H_i} \PP(\psi(X, Y, Z)\neq i)\ge \frac{1}{2}\left(1 - \tv(\PP_{0, XYZ}, \PP_{1, XYZ})\right) - \mu(\Gamma^c) - \mu(B_2(\epsilon)),
    \end{equation}
    where $\Gamma^c$ denotes the complement of set $\Gamma$, i.e. $\Gamma^c = \{\Gamma\in \RR^D\mid \theta\not\in \Gamma\}$, and $B_2(\epsilon)$ denotes the $\ell_2$-ball of radius $\epsilon$, i.e. $B_2(\epsilon) = \{\theta\in \RR^D: \|\theta\|_2\le \epsilon\}$. In the following proof, we choose distribution $\mu$ to be the following product of symmetric ternary distributions, i.e. for $\btheta = (\theta_1, \cdots, \theta_D)\in \Gamma$,
    $$\mu(\btheta) = \prod_{j=1}^D \mu_j(\theta_j)\quad \text{where }\mu_j = \begin{cases}
        (1-h)\cdot\delta_0 + \frac{h}{2}\cdot\delta_{r} + \frac{h}{2}\cdot\delta_{-r} &\quad \text{if } 1\le j\le d,\\
        \delta_0 &\quad \text{if }d+1\le j\le D,
    \end{cases}$$
    where $\delta_r$ denotes the point distribution at $r\in \RR$, and parameters $d\in [D]$, $h\in [0, 1]$ and $r > 0$ will be specified later. Then we have
    \begin{align*}
        \tv(\PP_{0, XYZ}, \PP_{1, XYZ})^2 & = \tv(\PP_{0, XZ}, \PP_{1, XZ})^2 \le \KL(\PP_{0, XZ}\ \|\ \PP_{1, XZ})\\
        & = \KL(\PP_{0, Z\mid X}\ \|\ \PP_{1, Z\mid X}\mid \PP_{0, X}) + \KL(\PP_{0, X}\ \|\ \PP_{1, X})\\
        & = \KL(\PP_{0, Z\mid X}\ \|\ \PP_{1, Z\mid X}\mid \PP_{0, X})\le \chi^2(\PP_{0, Z\mid X}\ \|\ \PP_{1, Z\mid X}\mid \PP_{0, X}). \numberthis\label{eq: tv-kl-chi-2}
    \end{align*}
    If we use $\varphi_{\btheta}(\cdot)$ to denote the density function of $\calN(\btheta, I_D)$, according to Ingster's trick \cite{ingster1987minimax} we have 
    \begin{align*} 
        &\hspace{-0.5cm} \chi^2(\PP_{0, Z\mid X}\ \|\ \PP_{1, Z\mid X}\mid \PP_{0, X}) + 1\\
        & = \EE_{\bX\sim \PP_{0, X}}\left[\EE_{\btheta\mid \bX, \btheta'\mid \bX}\left[\int_{(\RR^D)^m} \prod_{t=1}^m \frac{\varphi_\btheta(\bz^t)\varphi_{\btheta'}(\bz^t)}{\varphi_{\zero}(\bz^t)}d\bz_1\cdots d\bz_m\mid \bX\right]\right]\\
        & = \EE_{\bX\sim \PP_{0, X}}\left[\EE_{\btheta\mid \bX, \btheta'\mid \bX}[\exp\left(m\langle \btheta, \btheta'\rangle\right)\mid \bX]\right]\\
        & = \EE_{\bX\sim \PP_{0, X}}\left[\prod_{j=1}^d \left(\PP(\theta_j = \theta_j'\neq 0\mid \bX)(\exp(mr^2) - 1) + \PP(\theta_j = -\theta_j'\neq 0\mid X)(\exp(-mr^2) - 1) + 1\right)\right]\\
        & = \prod_{j=1}^d\EE_{\bX\sim \PP_{0, X}}\left[\PP(\theta_j = \theta_j'\neq 0\mid \bX)(\exp(mr^2) - 1) + \PP(\theta_j = -\theta_j'\neq 0\mid X)(\exp(-mr^2) - 1) + 1\right], \numberthis \label{eq: chi-square-lfht}
    \end{align*}
    where $\btheta, \btheta'$ are i.i.d. sampled according to $\PP(\btheta\mid \bX)$, and the last equation uses the fact that conditioned on $\bX$, we have $(\theta_j, \theta_j')$ independent to each other for any $j\in [D]$. In the following, we write $\bX = \bX^{1:n}$ which denotes the $n$ samples, and we further denote $\bX^i = (X^i_1, \cdots, X^i_D)$, where $X^i_j$ denotes the $j$-th coordinate of $\bX^i$. We notice that $\theta_j$ only depends on $X^{1:n}_j = (X^1_j, \cdots, X^n_j)$. According to Bayes rule, we can calculate 
    \begin{align*} 
        &\hspace{-0.5cm} \PP(\theta_j = 1\mid \bX)\\
        & = \frac{\Pr(X_j^{1:n}, \eta = 1)}{\Pr(X_j^{1:n}, \eta_j = 0) + \Pr(X_j^{1:n}, \eta_j = -1) + \Pr(X_j^{1:n}, \eta_j = 1)}\\
        & = \frac{h/2\cdot \prod_{i=1}^n\exp\left(-(X_j^i - r)^2/2\right)}{(1-h)\cdot \prod_{i=1}^n\exp\left(-(X_j^i)^2/2 \right) + h/2\cdot \prod_{i=1}^n\exp\left(- (X_j^i - r)^2/2\right) + h/2\cdot \prod_{i=1}^n\exp\left(-(X_j^i + r)^2/2\right)}\\
        & = \frac{h/2\cdot \exp\left(r\cdot \sum_{i=1}^n X_j^i\right)}{(1-h)\cdot \exp\left(r^2/2\right) + h/2\cdot \exp\left(r\cdot \sum_{i=1}^n X_j^i\right) + h/2\cdot \exp\left(-r\cdot \sum_{i=1}^n X_j^i\right)}. \numberthis \label{eq: positive-eta}
    \end{align*}
    Similarly, we get 
    \begin{equation}\label{eq: negative-eta}
        \PP(\theta_j = -1 \mid \bX) = \frac{h/2\cdot \exp\left(-r\cdot \sum_{i=1}^n X_j^i\right)}{(1-h)\cdot \exp\left(r^2/2\right) + h/2\cdot \exp\left(r\cdot \sum_{i=1}^n X_j^i\right) + h/2\cdot \exp\left(-r\cdot \sum_{i=1}^n X_j^i\right)}.
    \end{equation}
    In the following, we use $[0, 1]$-valued random variables $p_j(\bX)$ and $q_j(\bX)$ to denote 
    $$p_j(\bX) = \PP(\theta_j = 1\mid \bX)\quad\text{and}\quad q_j(\bX) = \PP(\theta_j = -1 \mid \bX),\qquad \forall j\in [D].$$
    We further notice that $\theta_j'$ and $\theta_j$ are i.i.d. conditioned on $\bX$. Hence we obtain 
    \begin{align*}
        &\hspace{-0.5cm} \PP(\theta_j = \theta_j'\neq 0\mid X)\left(\exp(mr^2) - 1\right) + \PP(\theta_j = -\theta_j'\neq 0\mid X)\left(\exp(-mr^2) - 1\right) + 1\\
        & = 1 + (p_j(\bX)^2 + q_j(\bX)^2)\left(\exp(mr^2) - 1\right) + 2p_j(\bX)q_j(\bX)\left(\exp(-mr^2) - 1\right)\\
        & = 1 + \frac{(p_j(\bX) + q_j(\bX))^2}{2}\cdot \left(\exp(mr^2) + \exp(-mr^2) - 2\right) + \frac{(p_j(\bX) - q_j(\bX))^2}{2}\cdot \left(\exp(mr^2) - \exp(-mr^2)\right). \numberthis \label{eq: p-eta-eta'}
    \end{align*}
    Next using the AM-GM inequality we obtain that
    $$(1-h)\cdot \exp\left(\frac{r^2}{2}
    \right) + \frac{h}{2}\cdot \exp\left(r\cdot \sum_{i=1}^n X_j^i\right) + \frac{h}{2}\cdot \exp\left(-r\cdot \sum_{i=1}^n X_j^i\right)\ge 1.$$
    Bringing this back to \pref{eq: positive-eta} and \pref{eq: negative-eta}, we obtain that
    \begin{align*} 
        &\hspace{-0.5cm} \left(p_j(\bX) + q_j(\bX)\right)^2\le \frac{h^2}{4}\cdot \left(\exp\left(r\cdot \sum_{i=1}^n X_j^i\right) + \exp\left(-r\cdot \sum_{i=1}^n X_j^i\right)\right)^2\\
        & \text{and}\quad \left(p_j(\bX) - q_j(\bX)\right)^2\le \frac{h^2}{4}\cdot \left(\exp\left(r\cdot \sum_{i=1}^n X_j^i\right) - \exp\left(-r\cdot \sum_{i=1}^n X_j^i\right)\right)^2. \numberthis \label{eq: p-q}
    \end{align*}
    We notice that according to the method of collecting samples, 
    $$X_j^i\iidsim \calN(\theta_j, 1)\quad\text{and}\quad \theta_j\sim (1-h)\cdot \delta_0 + \frac{\eta}{2}\cdot \delta_r + \frac{\eta}{2}\cdot \delta_{-r},$$
    which implies 
    $$\sum_{i=1}^n X_j^i\sim (1-h)\cdot \calN(0, n) + \frac{h}{2}\cdot \calN(nr, n) + \frac{h}{2}\cdot \calN(-nr, n).$$
    Next, we notice that for Gaussian random variable $X\sim \calN(\mu, \sigma)$, we have
    $$\EE[\exp (\alpha X)] = \exp\left(\alpha\mu + \frac{\alpha^2\sigma^2}{2}\right).$$
    Bringing this back to \pref{eq: p-q} we obtain that 
    \begin{align*} 
        \EE\left[\left(p_j(\bX) + q_j(\bX)\right)^2\right] & \le \frac{h^2}{4}\cdot \EE\left[\exp\left(2r\cdot \sum_{i=1}^n X_j^i\right) + 2 + \exp\left(-2r\cdot \sum_{i=1}^n X_j^i\right)\right]\\
        & = \frac{h^2}{2}\exp\left(2r^2n\right)\cdot \left(1-h + \frac{h}{2}\exp\left(2r^2n\right) + \frac{h}{2}\exp\left(-2r^2n\right)\right) + \frac{h^2}{2}\\
        \EE\left[\left(p_j(\bX) - q_j(\bX)\right)^2\right] & \le \frac{h^2}{4}\cdot \EE\left[\exp\left(2r\cdot \sum_{i=1}^n X_j^i\right) - 2 + \exp\left(-2r\cdot \sum_{i=1}^n X_j^i\right)\right]\\
        & = \frac{h^2}{2}\exp\left(2r^2n\right)\cdot \left(1-h + \frac{h}{2}\exp\left(2r^2n\right) + \frac{h}{2}\exp\left(-2r^2n\right)\right) - \frac{h^2}{2}.
    \end{align*}
    Hence if conditions 
    \begin{equation}\label{eq: m-n-less-1}
        mr^2\le 1\quad \text{and}\quad nr^2\le 1
    \end{equation}
    both hold, we have the following inequalities:
    \begin{align*}
        \EE\left[\left(p_j(\bX) + q_j(\bX)\right)^2\right] \le 30h^2,& \qquad\EE\left[\left(p_j(\bX) - q_j(\bX)\right)^2\right] \le 30h^2r^2n,\\
        \text{and} \qquad \exp(mr^2) + \exp(-mr^2) - 2 \le 2m^2r^4, & \qquad
        \exp(mr^2) - \exp(-mr^2) \le 3mr^2.
    \end{align*}
    Bringing them back to \pref{eq: p-eta-eta'}, we obtain that 
    \begin{align*}
        &\hspace{-0.5cm} \EE\left[\PP(\theta_j = \theta_j'\neq 0\mid X)\left(\exp(mr^2) - 1\right) + \PP(\theta_j = -\theta_j'\neq 0\mid X)\left(\exp(-mr^2) - 1\right) + 1\right]\\
        & \le 1 + 30h^2m^2r^4 + 45h^2mnr^4\le 1 + 75h^2mnr^4,
    \end{align*}
    where the last inequality uses the assumption $m < n$. Bring back to \pref{eq: chi-square-lfht}, we obtain that 
    \begin{align*}
        & \hspace{-0.5cm}\chi^2(\PP_{0, Z\mid X}\ \|\ \PP_{1, Z\mid X}\mid \PP_{0, X})]\\
        & \le \prod_{j=1}^d\EE\left[\PP(\theta_j = \theta_j'\neq 0\mid X)\left(\exp(mr^2) - 1\right) + \PP(\theta_j = -\theta_j'\neq 0\mid X)\left(\exp(-mr^2) - 1\right) + 1\right] - 1\\
        &\le (1 + 75h^2mnr^4)^d - 1.
    \end{align*}
    Hence according to \pref{eq: tv-kl-chi-2}, this implies that 
    \begin{equation}\label{eq: bound-on-tv}
        \tv(\PP_{0, XYZ}, \PP_{1, XYZ})^2\le \sqrt{(1 + 75h^2mnr^4)^d - 1}.
    \end{equation}

    Finally, we calculate the probability $\mu(\Gamma^c)$ and $\mu(B_2(\epsilon))$. According to our choice $d = d(\Gamma, n, \epsilon)$, when sampling $\btheta = (\theta_1, \cdots, \theta_D)\sim \mu$, with probability at least $1 - \delta$ we have 
    $$\sum_{t=1}^D \frac{|\theta_t|^p}{a_t^p} = \sum_{t=1}^d \frac{|\theta_t|^p}{a_{d}^p}\le \frac{dhr^p}{a_{d}^p} + \frac{r^p\cdot \sqrt{2d\log(2/\delta)}}{a_{d}^p},$$
    where the last inequality uses Hoeffding inequality. Additionally, when sampling $\btheta = (\theta_1, \cdots, \theta_D)\sim \mu$, with probability at least $1 - \delta$ we have
    $$\sum_{t=1}^D |\theta_t|^2 = \sum_{t=1}^d |\theta_t|^2 \ge dhr^2 - r^2\cdot \sqrt{2d\log(2/\delta)}.$$
    As long as $h^2d\ge 24$, with probability at least $9/10$ we have both
    $$\sum_{t=1}^D \frac{|\theta_t|^p}{a_t^p}\le \frac{2dhr^p}{a_{d}^p}\quad \text{and}\quad \sum_{t=1}^D |\theta_t|^2\ge \frac{1}{2}\cdot dhr^2.$$

    We choose $d = d_l(\Gamma, n, \epsilon)$ (here $d_l(\Gamma, n, \epsilon)$ is defined in \pref{eq: def-d-l}), then we have 
    $$(a_d)^p n^{\frac{p-2}{2}}\ge 192\epsilon^2.$$
    We further let 
    $$h = \frac{96\epsilon^2n}{d},\quad \text{and}\quad r = \frac{1}{\sqrt{n}}.$$
    According to the first inequality in \pref{eq: condition-m-n},  and also \pref{eq: condition-mn-lower}, we have $n\le d/(96\epsilon^2)$. This implies that $h\le 1$. Additionally, according to \pref{eq: condition-m-n} we have 
    $$h^2 = \frac{96\epsilon^4 n^2}{d^2}\ge \frac{96\epsilon^4}{d^2}\cdot \frac{d}{4\epsilon^4}\ge \frac{24}{d}.$$
    Hence with probability at least $9/10$ we have both
    $$\sum_{t=1}^D \frac{|\theta_t|^p}{a_t^p}\le \frac{2dhr^p}{a_d^p}\le 192\epsilon^2\cdot \frac{n^{\frac{2-p}{2}}}{a_d^p}\le 1,\quad \text{and}\quad \sum_{t=1}^D |\theta_t|^2\ge \frac{1}{2}dhr^2\ge \epsilon^2,$$
    which implies that 
    $$\mu(\Gamma^c) + \mu(B_2(\epsilon))\le \frac{1}{10}.$$
    Additionally, by our choice of $r$ and also \pref{eq: condition-m-n}, \pref{eq: m-n-less-1} always holds. Hence by \pref{eq: bound-on-tv}, 
    $$\tv(\PP_{0, XYZ}, \PP_{1, XYZ})^2\le \sqrt{(1 + 75h^2mnr^4)^d - 1}\le \frac{1}{10}.$$
    Therefore, according to \pref{eq: condition-lfht-testing-inequality}, we obtain that 
    $$\inf_{\psi} \max_{i\in \{0, 1\}} \sup_{P\in H_i} \PP(\psi(X, Y, Z)\neq i) > \frac{1}{4},$$
    which implies that \pref{eq: lfht-objective} fails.

    Above all, we have verify that we must have 
    $$mn\ge \frac{d_l(\Gamma, n, \epsilon)}{96\epsilon^4}$$
    in order to let \pref{eq: lfht-objective} satisfied.
\end{proof}

\subsection{Proof of \pref{thm: lfht-infinite-u}}\label{sec: app-lfht-infinite}
We present the proof of \pref{thm: lfht-infinite-u} in this section.
\begin{proof}[Proof of \pref{thm: lfht-infinite-u}]
    Suppose $\px = \calN(\btheta^x, I)$, $\py = \calN(\btheta^y, I)$ and $\pz = \calN(\btheta^z, I)$, and we let 
    $$\btheta^x = (\theta_{1:\infty}^x), \quad \btheta^y = (\theta_{1:\infty}^y),\quad \text{and}\quad \btheta^z = (\theta_{1:\infty}^z).$$
    We let $D = D(\Gamma, \epsilon)$, and we define $\bttheta^x, \bttheta^y, \bttheta^z\in \RR^D$ as 
    $$\bttheta^x = (\theta^x_1, \cdots, \theta^x_D), \quad \bttheta^y = (\theta^y_1, \cdots, \theta^y_D), \quad \text{and}\quad \bttheta^z = (\theta^z_1, \cdots, \theta^z_D),$$
    and we further define $D$-dimensional distributions 
    $$\tpx = \calN(\bttheta^x, I_D), \quad \tpy = \calN(\bttheta^y, I_D), \quad \text{and}\quad \tpz = \calN(\bttheta^z, I_D).$$
    According to the definition of $D$, we have for any $\btheta = (\theta_{1:\infty})\in \Theta$,
    $$\sum_{t=D+1}^\infty (\theta_t)^2\le (a_D)^2\cdot \sum_{t=1}^D \frac{|\theta_t|^2}{(a_t)^2}\stackrel{(i)}{\le} (a_D)^2\cdot \sum_{t=1}^D \frac{|\theta_t|^p}{(a_t)^p} \le (a_D)^2\stackrel{(ii)}{\le} \frac{\epsilon^2}{9},$$
    where $(i)$ uses the fact that $|\theta_t|/a_t\le 1$ for any $t$, and $(ii)$ uses \pref{eq: coor-ell-p-bodies}. Therefore, since
    $$\|\btheta_x - \btheta_y\|_2\ge \epsilon,$$
    we have 
    $$\|\bttheta_x - \bttheta_y\|_2\ge \|\btheta_x - \btheta_y\|_2 - \|\btheta_x - \bttheta_x\|_2 - \|\btheta_y - \bttheta_y\|_2\ge \epsilon - \frac{\epsilon}{3} - \frac{\epsilon}{3} = \frac{\epsilon}{3}.$$
    Hence via taking the testing scheme in \pref{sec: lfht-u} for the first $D$ coordinates, and also replacing $\epsilon$ with $\epsilon / 3$, we have that as long as $(m, n)$ satisfies
    $$\left\{(m, n): \qquad m\gtrsim \frac{1}{\epsilon^2},\quad n\gtrsim \frac{\sqrt{d_u(\Gamma, n, \epsilon)}}{\epsilon^2}\quad\text{and}\quad mn\gtrsim \frac{d_u(\Gamma, n, \epsilon)}{\epsilon^4}\right\},$$
    then $\psi$ is a feasible testing scheme which satisfies \pref{eq: lfht-objective}.
\end{proof}

\begin{proof}[Proof of \pref{thm: lfht-infinite-l}]
    We set $d = d_l(\Gamma, n, \epsilon)$. We consider the following subset of $\Gamma$: 
    $$\Gamma_d = \left\{\btheta = (\theta_{1:\infty}): \sum_{t=1}^d \frac{|\theta_t|^p}{(a_t)^p}\le 1,\quad \text{and}\quad \theta_t = 0,\ \  \forall t\ge d+1\right\} \subseteq \Gamma.$$
    Then if we can do likelihood-free hypothesis testing with $(m, n)$ samples for $\Gamma$, then we can also do likelihood-free hypothesis testing with $(m, n)$ samples for $\Gamma_d$. Notice that to do likelihood-free hypothesis testing, the data from coordinates greater than $d$ are completely independent to the $\px, \py$ and $\pz$. Hence without loss of generality we can assume that the model consists of dimension-$d$ distributions. Therefore, according to \pref{thm: lower-bound-lfht}, we get the desired result.
\end{proof}

\subsection{Missing Proofs in \pref{sec: lfht-theta}}\label{sec: app-lfht-theta}
\begin{proof}[Proof of \pref{prop: lfht-theta}]
    We notice that the $\Gamma$ defined in \pref{eq: def-theta} is an infinite dimensional $\ell_1$ body with $a_t = 1/t$. Therefore, we can calculate that 
    $$D(\Gamma, \epsilon) = \frac{1}{\epsilon},\quad \text{and}\quad d_l(\Gamma, n, \epsilon) \asymp d_u(\Gamma, n, \epsilon) \asymp \frac{1}{\sqrt{n}\epsilon^2},$$
    where in the above equations, we use $\asymp$ to hide constants and log factors of $n$ and $\epsilon$ as well. Therefore, according to \pref{thm: lfht-infinite-l} and \pref{thm: lfht-infinite-u}, we obtain the feasible region of likelihood-free hypothesis testing:
    $$\left\{(m, n): \quad m\gtrsim \epsilon^{-2},\quad n\gtrsim \epsilon^{-12/5},\quad mn^{3/2}\gtrsim \epsilon^{-6}\right\}.$$
\end{proof}

\end{document}